\pdfoutput=1

\newif\ifpersonal
\personaltrue

\documentclass[10pt,a4paper,dvipsnames,x11names]{amsart}

\DeclareMathAlphabet{\mathpzc}{OT1}{pzc}{m}{it}

\usepackage[utf8]{inputenc}
\usepackage[T1]{fontenc}
\usepackage[english]{babel}

\usepackage[a4paper,margin=1in]{geometry}

\usepackage{microtype,lmodern}
\usepackage{mathpazo}
\usepackage{euler}
\usepackage{mathtools}

\usepackage{amsmath,amssymb,amsthm,mathrsfs,bm,eucal,tensor}
\usepackage{aliascnt}
\usepackage{stmaryrd}
\usepackage{dsfont}

\usepackage[all,cmtip]{xy}
\usepackage{tikz}
\usetikzlibrary{arrows.meta,calc,positioning,fit,backgrounds,shapes.geometric}
\usepackage{tikz-cd}

\tikzset{
  >=Stealth,
  v/.style={draw, circle, inner sep=1.4pt},
  arr/.style={->, thick},
  darr/.style={->, thick, dashed}
}

\usepackage{enumerate,comment,braket,xspace,csquotes}
\usepackage[colorlinks=true,linkcolor=red,citecolor=red,urlcolor=red]{hyperref}

\newcommand{\one}{1\hskip-3.5pt1}

\newcommand{\C}{\mathcal C}
\newcommand{\Geom}{\mathrm{Geom}}
\newcommand{\pt}{\mathrm{pt}}
\newcommand{\id}{\mathrm{id}}
\newcommand{\Hom}{\mathrm{Hom}}
\newcommand{\End}{\mathrm{End}}

\newcommand{\forg}{\mathrm{forg}}
\newcommand{\Th}{\mathrm{Th}}

\newcommand{\Loc}{\mathrm{Loc}}
\newcommand{\BM}{\mathrm{BM}}
\newcommand{\D}{\mathbb D}
\providecommand{\Fix}{\operatorname{Fix}}
\providecommand{\OO}{\mathcal O}
\providecommand{\Perf}{\mathrm{Perf}}

\newcommand{\SLoc}{\operatorname{SLoc}}
\newcommand{\TLoc}{\mathfrak T}
\newcommand{\CSLoc}{\operatorname{CS}^{\mathrm{loc}}}
\newcommand{\SLocGpd}{\operatorname{SLocGpd}}

\numberwithin{equation}{section}
\theoremstyle{plain}
\newtheorem{theorem}{Theorem}[section]
\newaliascnt{proposition}{theorem}
\newtheorem{proposition}[proposition]{Proposition}
\aliascntresetthe{proposition}
\newaliascnt{lemma}{theorem}
\newtheorem{lemma}[lemma]{Lemma}
\aliascntresetthe{lemma}
\newaliascnt{corollary}{theorem}
\newtheorem{corollary}[corollary]{Corollary}
\aliascntresetthe{corollary}
\theoremstyle{definition}
\newaliascnt{definition}{theorem}
\newtheorem{definition}[definition]{Definition}
\aliascntresetthe{definition}
\theoremstyle{remark}
\newaliascnt{remark}{theorem}
\newtheorem{remark}[remark]{Remark}
\aliascntresetthe{remark}

 \title{A categorical and algebro-geometric theory of localization}
\author{Mauricio Corr\^ea}
\address{Dipartimento di Matematica, Universit\`a degli Studi di Bari Aldo Moro, Bari, Italy}
\email{mauricio.barros@uniba.it}
\author{Simone Noja}
\address{Dipartimento di Matematica, Universit\`a degli Studi di Bari Aldo Moro, Bari, Italy}
\email{simone.noja@uniba.it}
\date{}


\begin{document}

\begin{abstract}
We study cohomology classes whose restriction to the complement of a closed subset vanishes. The localisation sequence shows that their lifts to cohomology with support form a torsor under the image of the connecting morphism. We establish the behaviour of this torsor under excision, Cartesian base change, proper pushforward when compatible trace data are available, and external products, and characterise compatible local constructions in these terms. Under purity, excision and orientation hypotheses, its group of translations is identified with the image of the global-to-link Gysin transgression and hence with a subgroup of the kernel of multiplication by the Euler class; under a natural surjectivity hypothesis it is equal to this kernel. Injectivity of Euler multiplication therefore gives uniqueness of the supported lift, while concentration recovers the usual localisation formulae. In a stable categorical enhancement, we relate the set-level torsor to the corresponding space of nullhomotopies. For characteristic-class defects, we introduce Milnor localisation torsors and study their comparison morphisms and local index decompositions. For isolated hypersurface singularities, we describe the secondary ambiguity in terms of links and identify distinguished supported representatives arising from vanishing cycles.
\end{abstract}

\maketitle


\section{Introduction}\label{sec:intro}

In many topological and geometric situations, a global cohomological class is determined, or at least constrained, by its behaviour near a distinguished closed subset. Passing from a global class to a class with prescribed support is therefore a natural localisation problem. Such a supported refinement need not be unique, and the resulting ambiguity is the starting point of this paper. Fixed-point formulae in equivariant cohomology, the algebro-geometric localisation theorem, equivariant algebraic $K$-theory and virtual localisation are classical instances \cite{AB,BV,EG,Thomason,GP}; related concentration and index phenomena also occur in field-theoretic localisation \cite{PestunS4,PestunZabzine}.

Let
$
i:Z\hookrightarrow X,\  j:U=X\setminus Z\hookrightarrow X
$
be an open--closed decomposition, let $A$ be a coefficient object, and suppose that
$
c\in H^d(X;A),$ such that $ j^*c=0.
$
The vanishing of $j^*c$ is a primary support statement. It implies that $c$ admits a refinement supported on $Z$, but it does not in general select one. We therefore consider the secondary localisation torsor
\[
\SLoc_Z^d(c)
:=
\{\widetilde c\in H_Z^d(X;A)\mid \forg(\widetilde c)=c\}.
\]
By exactness it is a torsor under
\[
\TLoc_Z^d(A)
:=
\operatorname{im}\!\left(
\delta:H^{d-1}(U;j^*A)\longrightarrow H_Z^d(X;A)
\right).
\]
The support condition naturally gives this torsor. We determine its functorial and higher structure, identify its secondary ambiguity geometrically, and recover the usual localisation formulae after rigidification. 
After the adjunction $i_*\dashv i^!$, the torsor has an equivalent incarnation
\[
\Loc_Z^{\mathrm{tor}}(c)
\subset
\Hom_{\C(Z)}(\one_Z,i^!A[d]).
\]
A canonical localised class $\Loc_Z(c)$ exists precisely when this torsor is a singleton. If $\widetilde c_0,\widetilde c_1\in\SLoc_Z^d(c)$, their unique secondary displacement is
\[
\CSLoc_Z(\widetilde c_0,\widetilde c_1)
:=
\widetilde c_1-\widetilde c_0
\in
\TLoc_Z^d(A).
\]
The notation $\CSLoc$ reflects the formal analogy with Chern--Simons transgression: a primary vanishing leaves secondary choices, and the difference between two such choices is measured by a transgression class.

The secondary group also admits a geometric interpretation. Under the purity, excision, orientation and tubular-neighbourhood hypotheses of Subsection~\ref{subsec:boundary-link-transgression}, restriction from $U$ to the normal link followed by the Gysin boundary identifies $\TLoc_Z^d(A)$ with a subgroup of
\[
\ker\!\left(
\cup e(N):H^{d-2c}(Z;i^*A)\longrightarrow H^d(Z;i^*A)
\right).
\]
Equality holds when the relevant restriction to the link is surjective. In particular, injectivity of Euler multiplication forces the secondary group to vanish before any coefficient localisation, while concentration or Euler inversion recovers the usual denominator expressions.

The main results are as follows.
\begin{enumerate}[(i)]
\item \emph{Secondary localisation torsors.} Theorem~\ref{thm:Loc-factor} shows that the supported refinements of a class vanishing on the open complement form a nonempty torsor under the image of the localisation boundary, and identifies this torsor with its adjoint incarnation $\Loc_Z^{\mathrm{tor}}(c)$ on the closed stratum. Proposition~\ref{prop:secondary-translation-groupoid} records the associated translation groupoid and the secondary displacement $\CSLoc_Z$, while Remark~\ref{rem:higher-localization-torsor} describes the higher nullhomotopy space whose set-level quotient gives the localisation torsor.

\item \emph{Functoriality, characterisation, and local index.} Excision, Cartesian base change, proper pushforward with compatible trace data, and external products act naturally on the torsors in Propositions~\ref{prop:Loc-excision},~\ref{prop:Loc-bc},~\ref{prop:Loc-proper} and~\ref{prop:Loc-box}. Proposition~\ref{thm:Loc-universal} shows that compatible constructions of local terms determine points of the same supported-refinement torsor, while Theorem~\ref{thm:global-local} gives the corresponding local-to-global index decomposition for every chosen point of the torsor.

\item \emph{Link transgression and rigidification.} Proposition~\ref{prop:ambiguity-link-transgression} identifies the secondary transgression group with the image of global-to-link Gysin transgression and places it inside the Euler kernel, with equality under a natural surjectivity condition. Proposition~\ref{prop:canonicity-before-euler-inversion} gives the resulting pre-Euler canonicity criterion. Theorems~\ref{thm:self-int} and~\ref{thm:Loc-by-euler} give the self-intersection and Euler-division formulae, and Theorem~\ref{thm:universal-euler} gives their concentration form. The Borel and multiplicative $K$-theoretic cases are recovered in Theorem~\ref{thm:borel-conc}, Corollary~\ref{cor:abbv} and Theorem~\ref{thm:thomason-localisation}.

\item \emph{Milnor localisation torsors.} For a characteristic theory $C$ with a functorial class $C(X)$ and a virtual, expected or Fulton-type class $C^{\mathrm{vir}}(X)$ agreeing on the regular locus, the defect
\[
\Delta_C(X)=C^{\mathrm{vir}}(X)-C(X)
\]
is supported on $Z=\operatorname{Sing}(X)$. Section~\ref{sec:milnor-localization-torsors} associates with it the secondary Milnor-type invariant
\[
\mathfrak{Mil}_Z^C(X)=\SLoc_Z\bigl(\Delta_C(X)\bigr).
\]
Theorem~\ref{thm:milnor-ambiguity-torsor} identifies this as a torsor under the corresponding secondary transgression group, while Theorem~\ref{thm:milnor-characterisation} characterises compatible constructions of local Milnor classes as points of the same torsor. Proposition~\ref{prop:milnor-comparison} gives comparison morphisms between characteristic theories, and Corollary~\ref{cor:milnor-local-index} gives the local index decomposition. The torsor records all supported realisations of the defect before singularity-theoretic data select a distinguished representative.
\end{enumerate}

The comparison with secondary characteristic theory is structural. Chern--Simons forms, Cheeger--Simons characters, differential cohomology, bundle gerbes and string structures provide familiar passages from primary data to secondary choices or trivialisations \cite{ChernSimons,CheegerSimons,BrylinskiMcLaughlinI,BrylinskiMcLaughlinII,HopkinsSinger,WaldorfString}. Here that passage is already contained in open--closed recollement: the boundary morphism, and under geometric hypotheses link transgression, measures the secondary datum.

For Milnor classes, the characteristic-class defect and its local representatives belong to an established theory \cite{BrasseletSeadeSuwaVectorFields,ParusinskiPragaczJAG,YokuraMotivicMilnor,CappellMaximSchuermannShaneson,MaximSaitoYang,CallejasMorgadoSeadeChernClasses}. Here we keep track of the supported refinements, their acting group, and the additional geometric data required for canonicity.

\medskip
\noindent\textbf{Organisation of the paper.}
Sections~\ref{sec:ambient} and~\ref{sec:recollement} fix the coefficient theories and recollement sequence. Section~\ref{sec:UnivLoc} develops secondary localisation, its higher and functorial structures, link transgression and rigidification. Sections~\ref{sec:dual}--\ref{sec:concentration} treat duality, purity and concentration. Sections~\ref{sec:borel} and~\ref{sec:K} recover the equivariant cohomological and multiplicative $K$-theoretic forms of localisation. Section~\ref{sec:lefschetz} records Lefschetz-type constructions and standard six-functor realisations. Section~\ref{sec:milnor-localization-torsors} develops Milnor localisation torsors, isolated singularities, links and vanishing-cycle rigidification.

\section{Ambient data and coefficient theories}\label{sec:ambient}

\subsection{Spaces and morphisms}

\begin{definition} \label{def:Geom}
Fix a category $\Geom$ of spaces (schemes or stacks of finite type, complex analytic spaces,
smooth manifolds, Whitney-stratified spaces, etc.) in which we can form:
\begin{itemize}
\item closed immersions $i:Z\hookrightarrow X$ and open complements $j:U=X\setminus Z\hookrightarrow X$,
\item Cartesian squares,
\item proper morphisms.
\end{itemize}
\end{definition}

\subsection{Coefficient categories with six operations}

\begin{definition} \label{def:sixfun}
To each $X\in\Geom$ we attach a triangulated category $\C(X)$ equipped with
a symmetric monoidal structure $(\otimes,\one_X)$.
For each morphism $f:X\to Y$ we have exact functors
$$
f^*:\C(Y)\to\C(X),\qquad f_*:\C(X)\to\C(Y),
$$
with adjunction $f^*\dashv f_*$, and whenever invoked also functors
$$
f_!:\C(X)\to\C(Y),\qquad f^!:\C(Y)\to\C(X),
$$
with adjunction $f_!\dashv f^!$.
For open immersions $j$ we have $j_!\dashv j^*\dashv j_*$, and for closed immersions $i$ we have
$i^*\dashv i_*\dashv i^!$.

We assume:
\begin{itemize}
\item $f^*$ is symmetric monoidal: $f^*(A\otimes B)\simeq f^*A\otimes f^*B$ and $f^*\one_Y\simeq \one_X$;
\item (projection formula when used) $f_*(M\otimes f^*N)\simeq f_*M\otimes N$ and similarly for $f_!$;
\item (base change when used) for Cartesian squares we have Beck--Chevalley isomorphisms
for the relevant pairs among $(f^*,f_*),(f^*,f_!),(f^!,f_*),(f^!,f_!)$;
\item (K\"unneth/box product when used) a bifunctor $\boxtimes:\C(X)\times\C(Y)\to\C(X\times Y)$
compatible with pullbacks and proper pushforwards in the standard way.
\end{itemize}
\end{definition}

Throughout, we work in triangulated language, which gives a common language for the coefficient theories under consideration. All constructions are expressed in Hom-theoretic terms and therefore carry over, mutatis mutandis, to stable symmetric monoidal $\infty$-categories, upon replacing $\Hom$ by mapping spaces or mapping spectra.
 
\subsection{Coefficient rings and (co)homology}

\begin{definition}\label{def:ringobj}
Let $X\in\Geom$. An object $A\in\C(X)$ is a \emph{commutative ring object} if it is represented by a commutative algebra object in the chosen symmetric monoidal model or enhancement of $\C(X)$.  When such an enhancement is understood, we write $A\in \mathrm{CAlg}(\C(X))$.

\smallskip
In many standard models one fixes a commutative algebra object $B\in \mathrm{CAlg}(\C(\pt))$ on the point
(e.g.\ a field, a ring, a ring spectrum) and uses its pullback $p_X^*B$ as a coefficient object on $X$.
We will use both viewpoints: coefficients may live on $X$, and there are also \emph{ground scalars} coming from the point.
\end{definition}

\begin{definition}\label{def:coh}
For $A\in\C(X)$ define
$$
H^d(X;A):=\Hom_{\C(X)}(\one_X,\ A[d]).
$$
Let $p_X:X\to \pt$ be the structure map and set
\[
\mathscr R:=H^*(\pt;\one_{\pt})
=\bigoplus_{n\in\mathbb Z}\Hom_{\C(\pt)}(\one_{\pt},\one_{\pt}[n]).
\]
This is the \emph{graded-commutative ground ring}; its degree-zero part is
\[
R:=\mathscr R^0=\End_{\C(\pt)}(\one_{\pt}).
\]
If $B\in\C(\pt)$ is an object on the point, we also write
$$
H^d(X;B):=H^d\big(X;\ p_X^*B\big)=\Hom_{\C(X)}(\one_X,\ p_X^*B[d]).
$$
\end{definition}

\begin{lemma}\label{lem:R-action}
With $\mathscr R$ as in Definition~\ref{def:coh}, the total group
\[
H^*(X;A)=\bigoplus_d H^d(X;A)
\]
is canonically a graded $\mathscr R$-module, functorial in $A$ and under pullback in $X$.  In particular each $H^d(X;A)$ is an $R=\mathscr R^0$-module.
\end{lemma}

\begin{proof}
Let $p_X:X\to\pt$.  If $r\in\mathscr R^n$ and $\alpha\in H^d(X;A)$, pullback gives
\[
p_X^*r:\one_X\longrightarrow\one_X[n],
\]
and we set
\[
r\cdot\alpha:
\one_X\xrightarrow{p_X^*r}\one_X[n]
\xrightarrow{\alpha[n]}A[d+n].
\]
The graded module axioms follow from functoriality and the monoidal structure.  Naturality in $A$ is immediate, and for $f:Y\to X$ one has $f^*p_X^*r=p_Y^*r$ because $p_Y=p_X\circ f$.  Thus pullback is $\mathscr R$-linear in the graded sense.
\end{proof}

\begin{definition} \label{def:cup}
For $A,B\in\C(X)$ and classes $\alpha\in H^p(X;A)$, $\beta\in H^q(X;B)$, define
$$
\alpha\smile\beta\in H^{p+q}(X;A\otimes B)
$$
as the composite
$$
\one_X \xrightarrow{\simeq} \one_X\otimes \one_X
\xrightarrow{\alpha\otimes\beta} A[p]\otimes B[q]
\simeq (A\otimes B)[p+q].
$$
In particular $H^*(X;\one_X)$ is a graded ring.
\end{definition}

These hypotheses hold in the standard six-functor settings considered below. For constructible sheaves one may take $\C(X)=D^b_c(X;k)$ with unit $k_X$ \cite{KS,BBD}; for schemes or stacks one may use constructible $\ell$-adic complexes and their six operations \cite{SGA4,BBD,LO1,LO2}; motivic and rigid-analytic coefficient systems provide further examples \cite{AyoubGallauerVezzaniRigidAnalyticMotives}. Borel equivariant cohomology fits the same pattern, with graded ground ring $H_T^*(\pt)$, and will be used in Section~\ref{sec:borel}.

Equivariant Chow theory, virtual localisation and equivariant algebraic $K$-theory enter through their bivariant, obstruction-theoretic or multiplicative formulations. They exhibit the same distinction between a supported localisation problem and the additional geometric data which supply an Euler, virtual Euler or $\lambda_{-1}$ denominator \cite{EG,GP,Thomason}.


\subsection{Conventions on shifts and boundary morphisms}\label{subsec:shift-sign-conventions}

We use cohomological grading.  Distinguished triangles are written
\[
K\longrightarrow L\longrightarrow M\xrightarrow{+1},
\]
and all connecting morphisms in long exact sequences are those induced by this convention.  In particular, every boundary morphism denoted by \(\delta\) is obtained by applying \(\Hom(\one_X,-)\) to the relevant localisation triangle.

Whenever purity or a Thom isomorphism is invoked for a regular, oriented, or normally oriented closed immersion, the shifts and twists are those appearing in the chosen purity equivalence.  Euler classes and self-intersection morphisms are defined with respect to the same orientation data.  No additional sign convention is imposed beyond the standard signs of triangulated categories.

\section{Recollement and cohomology with supports}\label{sec:recollement}

\subsection{Recollement axioms for an open--closed pair}

Fix a closed immersion $i:Z\hookrightarrow X$ and open complement $j:U\hookrightarrow X$.

\begin{definition} \label{def:recollement}
We assume:
\begin{itemize}
\item adjunctions $i^*\dashv i_*\dashv i^!$ and $j_!\dashv j^*\dashv j_*$;
\item full faithfulness of $i_*$, $j_!$ and $j_*$;
\item vanishings $j^*i_*=0$, $i^*j_!=0$ and $i^!j_*=0$;
\item functorial distinguished triangles, for each $M\in\C(X)$:
$$
j_!j^*M \longrightarrow M \longrightarrow i_*i^*M \xrightarrow{+1},
\qquad
i_*i^!M \longrightarrow M \longrightarrow j_*j^*M \xrightarrow{+1}.
$$
\end{itemize}
\end{definition}

\subsection{Cohomology with supports and the localisation long exact sequence}

\begin{definition} \label{def:supports}
For $A\in\C(X)$ define
$$
H_Z^d(X;A):=\Hom_{\C(X)}(\one_X,\, i_*i^!A[d]).
$$
The forget-support map
$$
\forg:H_Z^d(X;A)\to H^d(X;A)
$$
is induced by the counit $i_*i^!A\to A$.
\end{definition}

\begin{theorem} \label{thm:les}
Assume the recollement hypotheses of Definition~\ref{def:recollement}. For every $A\in\C(X)$ there is a functorial long exact sequence
$$
\cdots \longrightarrow H_Z^{d}(X;A)\xrightarrow{\forg} H^{d}(X;A)\xrightarrow{j^{*}}
H^{d}(U;j^{*}A)\xrightarrow{\delta} H_Z^{d+1}(X;A)\longrightarrow \cdots.
$$
\end{theorem}

\begin{proof}
Apply $\Hom_{\C(X)}(\one_X,-)$ to the localisation triangle
\[
i_*i^!A\longrightarrow A\longrightarrow j_*j^*A\xrightarrow{+1}.
\]
The first two terms give $H_Z^d(X;A)$ and $H^d(X;A)$.  By $j^*\dashv j_*$ and $j^*\one_X\simeq\one_U$,
\[
\Hom_{\C(X)}(\one_X,j_*j^*A[d])
\simeq
H^d(U;j^*A).
\]
Under these identifications the first map is the forget-support morphism and the second is restriction to $U$.  The long exact sequence and its functoriality are therefore those attached to the distinguished triangle.
\end{proof}

\begin{lemma} \label{lem:supported-torsor}
Assume the recollement hypotheses of Definition~\ref{def:recollement}. Let $A\in\C(X)$ and let $\alpha\in H^{d}(X;A)$ satisfy
$j^{*}\alpha=0\in H^{d}(U;j^{*}A)$.
Then the set
\begin{equation*}
\forg^{-1}(\alpha)\ :=\ \{\widetilde\alpha\in H_Z^{d}(X;A)\mid \forg(\widetilde\alpha)=\alpha\}
\end{equation*}
is nonempty, and it is a principal homogeneous space under the subgroup
$$\mathrm{im}\big(\delta:H^{d-1}(U;j^{*}A)\to H_Z^{d}(X;A)\big).$$
In particular, the supported refinement is unique if and only if the image of this boundary morphism is zero, i.e.
\[
\operatorname{im}
\bigl(\delta:H^{d-1}(U;j^{*}A)\to H_Z^{d}(X;A)\bigr)=0.
\]
This holds, for example, if $H^{d-1}(U;j^{*}A)=0$.
\end{lemma}

\begin{proof}
Theorem~\ref{thm:les} gives $\ker(j^{*})=\mathrm{im}(\forg)$ at $H^{d}(X;A)$. Hence $j^{*}\alpha=0$ implies $\forg^{-1}(\alpha)\neq\varnothing$.
If $\widetilde\alpha,\widetilde\alpha'\in\forg^{-1}(\alpha)$, then
$\forg(\widetilde\alpha-\widetilde\alpha')=0$, so
$\widetilde\alpha-\widetilde\alpha'\in\ker(\forg)=\mathrm{im}(\delta)$
by exactness at $H_Z^{d}(X;A)$.
Conversely, if $\gamma\in H^{d-1}(U;j^{*}A)$ then
$\forg(\widetilde\alpha+\delta(\gamma))=\forg(\widetilde\alpha)$.
Thus $\mathrm{im}(\delta)$ acts freely and transitively on $\forg^{-1}(\alpha)$.
\end{proof}

\section{Secondary localisation: torsors, transgression, and functoriality}\label{sec:UnivLoc}

\subsection{Secondary localisation torsors and the localised class}

\begin{definition} \label{def:BM-rel}
Let $i:Z\hookrightarrow X$ be a closed immersion and let $A\in\C(X)$. Define
\[
A^{\BM}_m(Z\xrightarrow{i}X):=\Hom_{\C(Z)}(\one_Z,i^!A[-m]).
\]
\end{definition}

In geometric realisations this group is often identified with a relative Borel--Moore group, possibly after the relevant orientation data; the definition itself is independent of this identification.

\begin{definition}\label{def:LocZ}
Assume the recollement hypotheses of Definition~\ref{def:recollement}. Let $i:Z\hookrightarrow X$ be closed with open complement $j:U\hookrightarrow X$, let $A\in\C(X)$, and let $c\in H^d(X;A)$ satisfy $j^*c=0$.

The \emph{secondary localisation torsor} of $c$ along $Z$ is
\[
\SLoc_Z^d(c)
:=
\{\widetilde c\in H_Z^d(X;A)\mid \forg(\widetilde c)=c\}.
\]
Its acting \emph{secondary transgression group} is
\[
\TLoc_Z^d(A)
:=
\operatorname{im}\!\left(
\delta:H^{d-1}(U;j^*A)\longrightarrow H_Z^d(X;A)
\right).
\]
After adjunction $i_*\dashv i^!$, the same torsor is represented on $Z$ by
\[
\Loc_Z^{\mathrm{tor}}(c)
:=
\{\operatorname{adj}(\widetilde c)\in\Hom_{\C(Z)}(\one_Z,i^!A[d])\mid \widetilde c\in\SLoc_Z^d(c)\}.
\]
If this torsor is a singleton, its unique adjoint element is denoted by
\[
\Loc_Z(c)\in\Hom_{\C(Z)}(\one_Z,i^!A[d]).
\]
\end{definition}

\begin{theorem}\label{thm:Loc-factor}
Let $c\in H^d(X;A)$ satisfy $j^*c=0$. Then $\SLoc_Z^d(c)$ is a nonempty torsor under $\TLoc_Z^d(A)$, and adjunction identifies it with $\Loc_Z^{\mathrm{tor}}(c)$. In particular, a canonical localised class $\Loc_Z(c)$ is determined by the localisation triangle if and only if $\TLoc_Z^d(A)=0$.
\end{theorem}

\begin{proof}
This is Lemma~\ref{lem:supported-torsor} together with the adjunction $i_*\dashv i^!$.
\end{proof}

\begin{proposition}\label{prop:secondary-translation-groupoid}
Under the hypotheses above, set
\[
\CSLoc_Z(\widetilde c_0,\widetilde c_1)
:=
\widetilde c_1-\widetilde c_0
\in\TLoc_Z^d(A),
\qquad
\widetilde c_0,\widetilde c_1\in\SLoc_Z^d(c).
\]
Then
\[
\CSLoc_Z(\widetilde c,\widetilde c)=0,
\qquad
\CSLoc_Z(\widetilde c_1,\widetilde c_0)=-\CSLoc_Z(\widetilde c_0,\widetilde c_1),
\]
and
\[
\CSLoc_Z(\widetilde c_0,\widetilde c_2)
=
\CSLoc_Z(\widetilde c_0,\widetilde c_1)
+
\CSLoc_Z(\widetilde c_1,\widetilde c_2).
\]
Moreover,
\[
\SLocGpd_Z^d(c):=[\,\SLoc_Z^d(c)/\!/\TLoc_Z^d(A)\,]
\]
is a canonical translation groupoid whose unique morphism $\widetilde c_0\to\widetilde c_1$ is labelled by $\CSLoc_Z(\widetilde c_0,\widetilde c_1)$. It is connected with trivial automorphism groups and is abstractly, but not canonically, equivalent to the terminal groupoid.
\end{proposition}

\begin{proof}
The identities are immediate from the definition and the additive structure of $H_Z^d(X;A)$. The remaining statements are equivalent to the free and transitive action established in Theorem~\ref{thm:Loc-factor}.
\end{proof}

\begin{remark}\label{rem:functorial-secondary-ambiguity}
The notation $\CSLoc$ reflects the analogy with secondary transgression: it measures the displacement between two choices after a primary vanishing statement. Under the hypotheses of Subsection~\ref{subsec:boundary-link-transgression}, this displacement is interpreted through global-to-link transgression.
\end{remark}

\begin{remark}\label{rem:higher-localization-torsor}
In a stable symmetric monoidal $\infty$-categorical enhancement, mapping the localisation triangle gives a fibre sequence of mapping spaces
\[
\mathcal M_Z\longrightarrow\mathcal M_X\longrightarrow\mathcal M_U,
\]
where
\[
\mathcal M_Z=\operatorname{Map}(\one_X,i_*i^!A[d]),\quad
\mathcal M_X=\operatorname{Map}(\one_X,A[d]),\quad
\mathcal M_U=\operatorname{Map}(\one_U,j^*A[d]).
\]
Fix a point $\bar c\in\mathcal M_X$ representing $c$ and assume that $j^*c=0$ on $\pi_0$.  The space of nullhomotopies of the point $j^*\bar c\in\mathcal M_U$ is nonempty and, after choosing one nullhomotopy, is equivalent to $\Omega\mathcal M_U$.  Its connected components form a torsor under
\[
\pi_1\mathcal M_U\cong H^{d-1}(U;j^*A).
\]
Passing from a nullhomotopy to the induced supported cohomology class applies the boundary morphism $\delta$; two nullhomotopies differing by the image of $H^{d-1}(X;A)$ determine the same supported class.  Consequently
\[
\SLoc_Z^d(c)
\]
is the set-level quotient of this higher nullhomotopy torsor by that indeterminacy, and its acting group is precisely $\operatorname{im}\delta$.  For $k\geq1$, the higher homotopy groups of the nullhomotopy space are
\[
\pi_k\simeq H^{d-1-k}(U;j^*A).
\]
Thus the ordinary localisation torsor is the $0$-truncated shadow of the higher nullhomotopy problem, but it is not, in general, literally the $\pi_0$ of the homotopy fibre of $\mathcal M_Z\to\mathcal M_X$; compare \cite{LurieStableInfinity,CisinskiDeglise}.
\end{remark}

\subsection{Secondary transgression, links, and rigidification}\label{subsec:boundary-link-transgression}

The secondary indeterminacy is the image of the boundary morphism in the localisation sequence.  Under geometric hypotheses this boundary can be read on the link of the closed stratum, but one must retain the restriction from the global complement to that link.

Suppose that $i:Z\hookrightarrow X$ is a regular or normally oriented closed immersion of real codimension $2c$, with normal object $N\to Z$.  Assume a tubular neighbourhood $q:D(N)\to Z$, excision for cohomology with support, a Thom isomorphism, and the Gysin sequence for
\[
\pi:S(N)\longrightarrow Z.
\]
Assume moreover that the coefficient object is identified on the tubular neighbourhood with the pullback of its restriction to $Z$,
\[
A|_{D(N)}\simeq q^*i^*A,
\]
compatibly with the Thom and Gysin constructions.  Let
\[
r:H^{d-1}(U;j^*A)\longrightarrow H^{d-1}(S(N);\pi^*i^*A)
\]
be restriction to the boundary, followed by this coefficient identification.

\begin{proposition}\label{prop:ambiguity-link-transgression}
Under these hypotheses, the Thom isomorphism
\[
\Th_i:H^{d-2c}(Z;i^*A)\xrightarrow{\ \sim\ }H_Z^d(X;A)
\]
identifies the secondary transgression group with
\[
\Th_i^{-1}\bigl(\TLoc_Z^d(A)\bigr)
=
\operatorname{im}\!\left(
H^{d-1}(U;j^*A)
\xrightarrow{\ r\ }
H^{d-1}(S(N);\pi^*i^*A)
\xrightarrow{\ \pi_*\ }
H^{d-2c}(Z;i^*A)
\right).
\]
Consequently
\[
\Th_i^{-1}\bigl(\TLoc_Z^d(A)\bigr)
\subseteq
\ker\!\left(
\cup e(N):H^{d-2c}(Z;i^*A)\longrightarrow H^d(Z;i^*A)
\right).
\]
If $r$ is surjective in degree $d-1$ (more generally, if $\pi_*(\operatorname{im}r)=\operatorname{im}\pi_*$), then this inclusion is an equality.
\end{proposition}

\begin{proof}
Excision identifies $H_Z^d(X;A)$ with the supported cohomology of the disk-bundle neighbourhood $D(N)$.  Naturality of the localisation boundary gives a commutative diagram in which the global boundary
\[
H^{d-1}(U;j^*A)\xrightarrow{\delta}H_Z^d(X;A)
\]
is obtained by first restricting to $D(N)\setminus Z\simeq S(N)$ and then applying the local boundary.  Under the Thom identification, this local boundary is the Gysin map
\[
\pi_*:H^{d-1}(S(N);\pi^*i^*A)\longrightarrow H^{d-2c}(Z;i^*A);
\]
see, for example, \cite{BottTu}.  This proves the first formula.  Exactness of the Gysin sequence identifies $\operatorname{im}\pi_*$ with the kernel of multiplication by $e(N)$, yielding the inclusion and the stated equality criterion.
\end{proof}

\begin{proposition}\label{prop:canonicity-before-euler-inversion}
In the situation of Proposition~\ref{prop:ambiguity-link-transgression}, the secondary transgression group identifies, under Thom isomorphism, with a subgroup of the Euler-kernel term in the Gysin sequence.  Consequently, if multiplication by the Euler class
\[
\cup e(N):H^{d-2c}(Z;i^*A)\longrightarrow H^d(Z;i^*A)
\]
is injective in the relevant degree, then for every class \(c\in H^d(X;A)\) with \(j^*c=0\) the torsor \(\SLoc^d_Z(c)\) has a single element.
\end{proposition}

\begin{proof}
Proposition~\ref{prop:ambiguity-link-transgression} identifies the secondary transgression group acting on \(\SLoc^d_Z(c)\), after Thom isomorphism, with a subgroup of the kernel of multiplication by \(e(N)\).  If this multiplication map is injective, the subgroup is zero, and Theorem~\ref{thm:Loc-factor} gives uniqueness of the supported refinement.
\end{proof}

\begin{corollary}\label{cor:euler-inversion-rigidification}
In the situation of Proposition~\ref{prop:ambiguity-link-transgression}, suppose that, after passing to the chosen coefficient localisation, or to a periodic or twisted coefficient theory, multiplication by the Euler class \(e(N)\) is an isomorphism in the relevant degrees.  Then the link-transgression secondary group vanishes in the localised theory.  Consequently the localisation torsor has a single element, and the corresponding localised class is the Euler-divided class obtained from the purity and self-intersection results of Theorems~\ref{thm:self-int} and~\ref{thm:Loc-by-euler}.
\end{corollary}

\begin{proof}
Under the stated localisation or periodicity hypothesis, multiplication by \(e(N)\) has trivial kernel in the relevant degree.  By Proposition~\ref{prop:ambiguity-link-transgression}, the secondary transgression group \(\TLoc_Z^d(A)\) identifies with a subgroup of this kernel.  Hence \(\TLoc_Z^d(A)=0\), and Theorem~\ref{thm:Loc-factor} gives a unique supported refinement.  The identification of this refinement with the Euler-divided expression is the rigidification supplied by self-intersection and Euler division in Theorems~\ref{thm:self-int} and~\ref{thm:Loc-by-euler}.
\end{proof}

The localisation torsor identifies the global-to-link Gysin image with the secondary transgression group governing supported localisation. Classical link transgression therefore describes the ambiguity left unresolved by the localisation triangle.  In de Rham-type models this boundary is represented by Thom-form transgression \cite{MathaiQuillen,BottTu}.  In the concentration situations considered below, the prescribed localisation of coefficients supplies the denominator and makes the relevant secondary transgression group vanish.

\begin{remark}\label{rem:quadratic-canonicity-obstruction}
In motivic theories endowed with an \(SL_{\eta}\)-orientation, the Euler class appearing above is replaced by the \(SL_{\eta}\)-oriented Euler class, with its natural determinant twist.  Injectivity of multiplication by this Euler class therefore still gives a canonicity criterion.  Failure of injectivity only leaves room for secondary ambiguity in general; it becomes an actual obstruction when the equality criterion in Proposition~\ref{prop:ambiguity-link-transgression} is satisfied.  This is complementary to the rigidified quadratic localisation formulae of Hoyois, Levine and D'Angelo \cite{HoyoisTrace,LevineWitt,DAngeloSLeta}.
\end{remark}

\begin{remark}\label{rem:secondary-analogies}
The passage from $j^*c=0$ to the torsor $\SLoc_Z^d(c)$ is analogous to familiar secondary constructions in topology and differential geometry. A primary vanishing or trivialisation may leave a family of choices, while the difference between two choices is measured by a transgression class. Chern--Simons invariants, differential characters, gerbes and string structures provide classical manifestations of this principle \cite{ChernSimons,CheegerSimons,BrylinskiMcLaughlinI,BrylinskiMcLaughlinII,HopkinsSinger,WaldorfString}. In the present situation the secondary datum arises directly from the open--closed localisation sequence, and
\[
\CSLoc_Z(\widetilde c_0,\widetilde c_1)
=
\widetilde c_1-\widetilde c_0
\]
measures the displacement between two supported refinements.

A related distinction occurs in localisation problems in topology and mathematical physics. Additional geometric data may single out a preferred local contribution, whereas the torsor records the ambiguity present before such a choice is made. Thus the secondary localisation torsor and its transgression group belong to the topology of the localisation problem itself, while a rigidification selects a distinguished point of the torsor.
\end{remark}

\subsection{Functoriality axioms used by \texorpdfstring{$\Loc_Z$}{LocZ}}

\begin{definition} \label{def:Loc-func-axioms}
In the rest of this section we use the following compatibilities whenever stated:
\begin{itemize}
\item \textbf{(BC)} Beck--Chevalley isomorphisms for Cartesian squares involving a closed immersion $i$ and its pullback $i'$:
$$
g^*i_*\simeq i'_*g_Z^*,
\qquad
g_Z^*i^!\simeq i'^!g^*,
$$
and similarly with $j$-functors for open complements.
\item \textbf{(PF)} Projection formula for proper pushforwards and for closed immersions whenever needed.
\item \textbf{(Tr)} Whenever a proper pushforward on the cohomology groups $H^*(X;p^*A)$ is used, a natural trace morphism
\[
\operatorname{tr}_{p,A}:p_*p^*A\longrightarrow A
\]
is part of the data, compatible with composition and with the base-change maps being invoked.  Such a trace may, for example, arise from an orientation $p^*A\to p^!A$ followed, for proper $p$, by the counit $p_*p^!A\simeq p_!p^!A\to A$.  If the natural trace carries a shift or a twist, the same statements hold after incorporating it into the degree and coefficient notation.
\item \textbf{(Ext)} Existence and compatibility of the bifunctor $\boxtimes$.
\end{itemize}
\end{definition}

\subsection{Excision}

\begin{proposition} \label{prop:Loc-excision}
Assume the recollement hypotheses of Definition~\ref{def:recollement}. Let $v:V\hookrightarrow X$ be an open neighbourhood of $Z$ and write
$i_V:Z\hookrightarrow V$ for the induced closed immersion and $j_V:V\setminus Z\hookrightarrow V$ for its open complement.
Assume (BC) for the square determined by $v$ and $i$ (so that $v^*i_*\simeq i_{V*}$ and $i_V^!v^*\simeq v_Z^*i^!$).
Let $A\in\C(X)$, fix $d\in\mathbb Z$, and let $c\in H^d(X;A)$ satisfy $j^*c=0$.
Then $j_V^*(v^*c)=0$ and, under the canonical identification
$$
\Hom_{\C(Z)}(\one_Z,i^!A[d])\ \simeq\ \Hom_{\C(Z)}(\one_Z,i_V^!v^*A[d]),
$$
one has an equality of torsors
$$
\Loc_Z^{\mathrm{tor},X}(c)\ =\ \Loc_Z^{\mathrm{tor},V}(v^*c).
$$
In particular, if either side is a singleton, then so is the other and $\Loc_Z^{X}(c)=\Loc_Z^{V}(v^*c)$.
\end{proposition}

\begin{proof}
The vanishing is immediate from functoriality:
$$
j_V^*(v^*c)= (j^*c)|_{V\setminus Z}=0.
$$

For cohomology with supports, adjunction identifies
\begin{align}
H_Z^d(X;A)
&=\Hom_{\C(X)}(\one_X,i_*i^!A[d])
\cong \Hom_{\C(Z)}(i^*\one_X,i^!A[d])
\cong \Hom_{\C(Z)}(\one_Z,i^!A[d]), \label{eq:excision-adj-X}\\
H_Z^d(V;v^*A)
&=\Hom_{\C(V)}(\one_V,i_{V*}i_V^!v^*A[d])
\cong \Hom_{\C(Z)}(i_V^*\one_V,i_V^!v^*A[d])
\cong \Hom_{\C(Z)}(\one_Z,i_V^!v^*A[d]). \label{eq:excision-adj-V}
\end{align}
By Beck--Chevalley for the square determined by $v$ and $i$, one has a canonical isomorphism
\begin{equation}\label{eq:excision-BC-identification}
i_V^!v^*A\ \simeq\ i^!A.
\end{equation}
Combining \eqref{eq:excision-adj-X}, \eqref{eq:excision-adj-V}, and \eqref{eq:excision-BC-identification} yields a canonical isomorphism
\begin{equation}\label{eq:excision-supports-iso}
H_Z^d(X;A)\xrightarrow{\ \sim\ } H_Z^d(V;v^*A).
\end{equation}
Since the forget-support maps are induced by the counits $i_*i^!\Rightarrow \id$ and $i_{V*}i_V^!\Rightarrow \id$, and Beck--Chevalley is compatible with these counits, the isomorphism \eqref{eq:excision-supports-iso} carries supported refinements of $c$ bijectively to supported refinements of $v^*c$. Thus
$$
\SLoc_Z^d(c)\xrightarrow{\ \sim\ }\SLoc_Z^d(v^*c),
$$
and after adjunction this identifies the localisation torsors
$$
\Loc_Z^{\mathrm{tor},X}(c)
\ =\
\Loc_Z^{\mathrm{tor},V}(v^*c)
$$
under the stated identification of targets.
Finally, if either torsor is a singleton, the equality of torsors forces equality of the unique element, giving
$$
\Loc_Z^{X}(c)=\Loc_Z^{V}(v^*c).
$$
\end{proof}

\subsection{Cartesian base change}

\begin{proposition} \label{prop:Loc-bc}
Assume the recollement hypotheses of Definition~\ref{def:recollement}. Consider a Cartesian square
\begin{equation}\label{eq:bc-square-closed}
\begin{tikzcd}[row sep=large, column sep=large]
Z' \arrow[r,"g_Z"] \arrow[d,"i'"'] & Z \arrow[d,"i"]\\
X' \arrow[r,"g"'] & X
\end{tikzcd}
\end{equation}
and write $j:U\hookrightarrow X$, $j':U'\hookrightarrow X'$ for the open complements.
Let $A\in\C(X)$ and $d\in\mathbb Z$.
Assume Beck--Chevalley holds for the closed square \eqref{eq:bc-square-closed} and for the induced open square
\begin{equation*}
\begin{tikzcd}[row sep=large, column sep=large]
U' \arrow[r,"g_U"] \arrow[d,"j'"'] & U \arrow[d,"j"]\\
X' \arrow[r,"g"'] & X
\end{tikzcd}
\end{equation*}
so that we have canonical isomorphisms
$$
g^*i_*\simeq i'_*g_Z^*,\qquad g_Z^*i^!\simeq i'^!g^*,
\qquad\text{and}\qquad
j'^*g^*\simeq g_U^*j^*.
$$
Let $c\in H^d(X;A)$ satisfy $j^*(c)=0$. Then $j'^*(g^*c)=0$ and, after identifying
$g_Z^*i^!A\simeq i'^!g^*A$ by base change, pullback induces a natural morphism of torsors
$$
g_Z^*:\Loc_Z^{\mathrm{tor}}(c)\longrightarrow \Loc_{Z'}^{\mathrm{tor}}(g^*c)
\qquad\text{inside}\qquad
\Hom_{\C(Z')}(\one_{Z'},\,i'^!g^*A[d]).
$$
In particular, if $\Loc_Z(c)$ and $\Loc_{Z'}(g^*c)$ are defined (i.e.\ the corresponding torsors are singletons), then
$$
\Loc_{Z'}(g^*c)\ =\ g_Z^*\Loc_Z(c)
\qquad\text{in}\qquad
\Hom_{\C(Z')}(\one_{Z'},\,i'^!g^*A[d]).
$$
\end{proposition}

\begin{proof}
The vanishing is immediate from Beck--Chevalley for the open square:
$$
j'^*(g^*c)\ \simeq\ g_U^*(j^*c)\ =\ 0.
$$
Let $\widetilde c\in \SLoc_Z^d(c)$, so that
$$
\widetilde c:\one_X\longrightarrow i_*i^!A[d]
\qquad\text{and}\qquad
\epsilon_{A[d]}\circ \widetilde c=c.
$$
Applying $g^*$ and using $g^*\one_X\simeq \one_{X'}$, together with Beck--Chevalley for the closed square, gives a morphism
$$
g^*\widetilde c:\one_{X'}\longrightarrow g^*(i_*i^!A)[d]\simeq i'_*i'^!g^*A[d].
$$
Naturality of Beck--Chevalley with respect to counits implies that under these identifications the pullback of
$\epsilon_{A[d]}:i_*i^!A[d]\to A[d]$
corresponds to the counit
$\epsilon_{g^*A[d]}:i'_*i'^!g^*A[d]\to g^*A[d]$.
Hence
$$
\epsilon_{g^*A[d]}\circ g^*\widetilde c
\ \simeq\
g^*(\epsilon_{A[d]})\circ g^*(\widetilde c)
\ =\
g^*(\epsilon_{A[d]}\circ \widetilde c)
\ =\
g^*c,
$$
so $g^*\widetilde c\in \SLoc_{Z'}^d(g^*c)$. This construction is functorial in $\widetilde c$, hence yields a natural map
$$
g^*: \SLoc_Z^d(c)\longrightarrow \SLoc_{Z'}^d(g^*c).
$$
Passing to adjoints under $i_*\dashv i^!$ and $i'_*\dashv i'^!$ gives the announced morphism of localisation torsors
$$
g_Z^*:\Loc_Z^{\mathrm{tor}}(c)\longrightarrow \Loc_{Z'}^{\mathrm{tor}}(g^*c).
$$
If both torsors are singletons, then the image of the unique element of $\Loc_Z^{\mathrm{tor}}(c)$ must be the unique element of $\Loc_{Z'}^{\mathrm{tor}}(g^*c)$, which is precisely the stated compatibility of canonical localised classes.
\end{proof}

\subsection{Proper pushforward}

\begin{proposition}\label{prop:Loc-proper}
Assume Definition~\ref{def:recollement}.  Let $p:X\to Y$ be proper, let $i:Z\hookrightarrow X$ and $k:W\hookrightarrow Y$ be closed immersions, and assume $p(Z)\subset W$.  Write
\[
p_Z:Z\longrightarrow W
\]
for the induced proper morphism, and let $j:U=X\setminus Z\hookrightarrow X$ and $\ell:V=Y\setminus W\hookrightarrow Y$ be the open complements.  Assume proper base change over $V$ and a trace morphism
\[
\operatorname{tr}_{p,A}:p_*p^*A\longrightarrow A
\]
as in \textbf{(Tr)}, compatible with that base change.

For $c:\one_X\to p^*A[d]$, define its proper cohomological pushforward $p_*c\in H^d(Y;A)$ by
\[
\one_Y\longrightarrow p_*\one_X\xrightarrow{p_*c}p_*p^*A[d]
\xrightarrow{\operatorname{tr}_{p,A}[d]}A[d],
\]
where the first arrow is the unit $\one_Y\to p_*p^*\one_Y\simeq p_*\one_X$.  If $j^*c=0$, then $\ell^*(p_*c)=0$, and proper pushforward induces a natural morphism of localisation torsors
\[
(p_Z)_*^{\mathrm{loc}}:
\Loc_Z^{\mathrm{tor}}(c)
\longrightarrow
\Loc_W^{\mathrm{tor}}(p_*c).
\]
On a supported representative $\widetilde c:\one_X\to i_*i^!p^*A[d]$, the induced representative on $Y$ is the composite
\[
\one_Y
\longrightarrow p_*\one_X
\longrightarrow p_*i_*i^!p^*A[d]
\simeq k_*(p_Z)_*i^!p^*A[d]
\longrightarrow k_*k^!A[d],
\]
where the last arrow is adjoint to
\[
k_*(p_Z)_*i^!p^*A[d]
\simeq p_*i_*i^!p^*A[d]
\longrightarrow p_*p^*A[d]
\xrightarrow{\operatorname{tr}_{p,A}[d]} A[d].
\]
Surjectivity requires additional hypotheses on the induced map of secondary transgression groups.  If both torsors are singletons, then
\[
\Loc_W(p_*c)=(p_Z)_*^{\mathrm{loc}}\Loc_Z(c).
\]
\end{proposition}

\begin{proof}
Since $p(Z)\subset W$, one has $p^{-1}(V)\subset U$.  Proper base change, together with compatibility of the trace with restriction to $V$, identifies $\ell^*(p_*c)$ with the proper pushforward of $c|_{p^{-1}(V)}$.  This restriction is zero because $j^*c=0$, hence $\ell^*(p_*c)=0$.

Let $\widetilde c$ be a supported refinement of $c$.  The displayed construction produces a class in $H_W^d(Y;A)$.  Composing it with the forget-support morphism $k_*k^!A[d]\to A[d]$ gives, by naturality of the counit and the chosen trace, the cohomological pushforward $p_*c$.  Hence it is a supported refinement of $p_*c$.  The construction is functorial in $\widetilde c$ and affine with respect to the induced homomorphism of secondary groups, so after adjunction it gives the asserted morphism of torsors.  The singleton statement is immediate.
\end{proof}

\subsection{Compatibility with \texorpdfstring{$\boxtimes$}{⊠}}

\begin{definition}\label{def:ext-prod-obj}
For $X,Y$ assume we are given a bifunctor
\begin{equation*}
\begin{tikzcd}[column sep=large]
\C(X)\times \C(Y) \arrow[r, "\boxtimes"] &
\C(X\times Y)
\end{tikzcd}
\end{equation*}
which is bi-exact (triangulated models) and compatible with pullbacks and pushforwards in the usual six-functor sense
(whenever those compatibilities are invoked later).
\end{definition}

\begin{definition}\label{def:ext-prod-coh}
For $A\in \C(X)$, $B\in \C(Y)$ and classes
$\alpha\in H^p(X;A)=\Hom_{\C(X)}(\one_X,A[p])$,
$\beta\in H^q(Y;B)=\Hom_{\C(Y)}(\one_Y,B[q])$,
define
$$
\alpha\boxtimes\beta\in H^{p+q}(X\times Y;A\boxtimes B)
$$
as the composite
\begin{equation*}
\begin{tikzcd}[column sep=large]
\one_{X\times Y}
\arrow[r, "\sim"] &
\one_X\boxtimes\one_Y
\arrow[r, "\alpha\boxtimes\beta"] &
A[p]\boxtimes B[q]
\arrow[r, "\sim"] &
(A\boxtimes B)[p+q]
\end{tikzcd}
\end{equation*}
\end{definition}

\begin{proposition} \label{prop:Loc-box}
Let $i:Z\hookrightarrow X$ be closed with open complement $j:U\hookrightarrow X$, and let $Y$ be any space.
Write $I:Z\times Y\hookrightarrow X\times Y$ for the product immersion and $J:U\times Y\hookrightarrow X\times Y$
for its open complement. Let $A\in \C(X)$, $B\in \C(Y)$, let $c\in H^d(X;A)$ satisfy $j^*c=0$, and let
$\beta\in H^e(Y;B)$.

Assume that the external product is compatible with closed pushforward and extraordinary pullback for $I$, in the sense that there are functorial isomorphisms
\begin{equation*}
(i_*C)\boxtimes B\simeq I_*(C\boxtimes B),
\qquad
I^!(A\boxtimes B)\simeq i^!A\boxtimes B,
\end{equation*}
compatible with the corresponding counits. Then $J^*(c\boxtimes\beta)=0$, and external product with $\beta$ induces a natural morphism of torsors
\begin{equation}\label{eq:Loc-box-map}
\Loc_Z^{\mathrm{tor}}(c)
\longrightarrow
\Loc_{Z\times Y}^{\mathrm{tor}}(c\boxtimes\beta),
\qquad
\lambda\longmapsto \lambda\boxtimes\beta.
\end{equation}
Surjectivity requires further K\"unneth-type hypotheses on the secondary transgression groups.  In particular, whenever the two torsors are singletons, one obtains
\begin{equation*}
\Loc_{Z\times Y}(c\boxtimes \beta)=\Loc_Z(c)\boxtimes \beta.
\end{equation*}
\end{proposition}

\begin{proof}
The vanishing of $J^*(c\boxtimes\beta)$ follows from functoriality of pullback and the identity
$J^*(c\boxtimes\beta)=(j^*c)\boxtimes\beta=0$.
Let $\widetilde c:\one_X\to i_*i^!A[d]$ be a supported refinement of $c$.  Forming the external product with $\beta$ and using the stated compatibility gives a morphism
\begin{equation*}
\one_{X\times Y}
\longrightarrow
(i_*i^!A[d])\boxtimes B[e]
\simeq
I_*((i^!A)\boxtimes B)[d+e].
\end{equation*}
Compatibility with the counits identifies its image under the forget-support map with $c\boxtimes\beta$.  Thus it is a supported refinement of $c\boxtimes\beta$ along $Z\times Y$.  Passing to adjoints under $I_*\dashv I^!$ and using $I^!(A\boxtimes B)\simeq i^!A\boxtimes B$ gives the morphism of torsors \eqref{eq:Loc-box-map}.  If both torsors are singletons, this morphism carries the unique element of the source to the unique element of the target, which is the stated identity of canonical localised classes.
\end{proof}

\subsection{Characterisation by supported refinements}

\begin{proposition} \label{thm:Loc-universal}
Fix the ambient six-functor setting and a closed immersion $i:Z\hookrightarrow X$ with open complement $j:U\hookrightarrow X$.
Let $\Lambda_Z$ assign, to every $A\in\C(X)$ and every $c\in H^d(X;A)$ with $j^*c=0$, an element
\[
\Lambda_Z(c)\in \Hom_{\C(Z)}(\one_Z,i^!A[d]).
\]
If the adjoint $\widetilde\Lambda_Z(c):\one_X\to i_*i^!A[d]$ satisfies
\[
\epsilon_{A[d]}\circ\widetilde\Lambda_Z(c)=c,
\]
then
\[
\Lambda_Z(c)\in\Loc_Z^{\mathrm{tor}}(c).
\]
If, in addition, $\Lambda$ is compatible with excision, base change, proper pushforward or external products, these compatibilities agree with the corresponding natural maps of localisation torsors constructed above.  In particular, if $\Loc_Z^{\mathrm{tor}}(c)$ is a singleton, then
\[
\Lambda_Z(c)=\Loc_Z(c).
\]
\end{proposition}

\begin{proof}
The displayed compatibility says exactly that $\widetilde\Lambda_Z(c)$ is a supported refinement of $c$.  Hence it lies in $\SLoc_Z^d(c)$, and adjunction places $\Lambda_Z(c)$ in $\Loc_Z^{\mathrm{tor}}(c)$.  The functorial statements follow from Propositions~\ref{prop:Loc-excision},~\ref{prop:Loc-bc},~\ref{prop:Loc-proper} and~\ref{prop:Loc-box}.  If the torsor is a singleton, membership forces $\Lambda_Z(c)=\Loc_Z(c)$.
\end{proof}

\section{Duality, orientations, and local-to-global index formulas}\label{sec:dual}

\subsection{Verdier duality (axiomatic)}

\begin{definition} \label{def:duality}
Assume that for each $X$ we are given a \emph{dualizing object} $\omega_X\in\C(X)$ and an internal Hom functor
$\underline{\Hom}(-,-)$ (or a model-specific derived internal Hom) so that Verdier duality is the contravariant functor
$$
\D_X(M):=\underline{\Hom}(M,\omega_X).
$$
We use only the formal consequences needed to speak about trace/pairing maps when they exist in the chosen model.
\end{definition}

\subsection{Index formulas supported on \texorpdfstring{$Z$}{Z}}

\begin{definition} \label{def:global-index}
Let $p_X:X\to \pt$ be proper and let $A\in\C(\pt)$.
For $d\in\mathbb Z$ define the global index map in degree $d$ whenever a proper cohomological trace, as in \textbf{(Tr)}, is available
by
$$
\int_X(-)\ :=\ (p_X)_*:\ H^d(X;p_X^*A)\longrightarrow H^d(\pt;A)=\Hom_{\C(\pt)}(\one_\pt,A[d]).
$$
When $d=0$ and $A=\one_\pt$ this takes values in $R=\End_{\C(\pt)}(\one_\pt)$.
\end{definition}

\begin{definition}\label{def:local-index}
Let $i:Z\hookrightarrow X$ be closed and let $p_X:X\to\pt$ be proper.  Assume the functoriality needed for Proposition~\ref{prop:Loc-proper}.  For a chosen point
\[
\ell\in\Loc_Z^{\mathrm{tor}}(c)\subset
\Hom_{\C(Z)}(\one_Z,i^!p_X^*A[d]),
\]
whenever the model identifies $\ell$ with a class on $Z$ admitting proper pushforward, define
\[
\int_Z\ell:=(p_Z)_*(\ell)\in H^d(\pt;A),
\qquad p_Z=p_X\circ i.
\]
Via the adjunction between supported representatives and points of the torsor, we use the same integral notation for either form. If the torsor is a singleton, we write this as $\int_Z\Loc_Z(c)$.
\end{definition}

\begin{theorem}\label{thm:global-local}
Under the hypotheses of Definition~\ref{def:local-index}, let $c\in H^d(X;p_X^*A)$ satisfy $j^*c=0$.  Then every
\[
\ell\in\Loc_Z^{\mathrm{tor}}(c)
\]
has the same proper pushforward, namely
\[
\int_Xc=\int_Z\ell.
\]
If $Z=\coprod_\lambda Z_\lambda$ is a finite disjoint union and $\ell=(\ell_\lambda)_\lambda$ is the corresponding decomposition of the chosen supported refinement, then
\[
\int_Xc=\sum_\lambda\int_{Z_\lambda}\ell_\lambda.
\]
\end{theorem}

\begin{proof}
Apply Proposition~\ref{prop:Loc-proper} to the structure morphism $p_X:X\to\pt$.  Localisation along the identity of the point has a single supported refinement, so the proper pushforward of every point of $\Loc_Z^{\mathrm{tor}}(c)$ is $p_{X*}(c)$.  This is the first identity.  For a finite disjoint union, $i_*i^!$ decomposes as the finite direct sum of the corresponding $i_{\lambda*}i_\lambda^!$, and proper pushforward is additive, giving the second formula.
\end{proof}

\section{Purity and Euler-denominator formulae}\label{sec:pure}

\subsection{Purity, orientation, and invertible Euler classes}

\begin{definition}\label{def:purity-axiom}
Let $i:Z\hookrightarrow X$ be a regular immersion of codimension $c$.
An \emph{oriented purity structure for $i$} consists of:
\begin{itemize}
\item a purity isomorphism
\[
\rho_i:i^!\one_X\xrightarrow{\ \sim\ }\one_Z[-2c];
\]
\item for every $A\in\C(X)$, a functorial $i^!$-linearity isomorphism
\[
\mu_A:i^!\one_X\otimes i^*A\xrightarrow{\ \sim\ }i^!A;
\]
\item the projection-formula and monoidal compatibilities needed for cup products and Gysin maps.
\end{itemize}
The shifts are to be replaced by the corresponding twists in theories in which purity carries a Tate or determinant twist.
\end{definition}

\begin{definition}\label{def:thom-euler}
Assume Definition~\ref{def:purity-axiom}.  The \emph{Thom class}
\[
u(i)\in H_Z^{2c}(X;\one_X)
\]
is the unique supported class whose adjoint
\[
\one_Z\longrightarrow i^!\one_X[2c]
\]
becomes $\id_{\one_Z}$ after composition with $\rho_i[2c]$.
Its image in ordinary cohomology is the Gysin class $i_*(1_Z)$, and the \emph{Euler class} is
\[
e(i):=i^*i_*(1_Z)\in H^{2c}(Z;\one_Z).
\]
\end{definition}

\subsection{Thom operators and self-intersection}

\begin{definition}\label{def:Th-operator}
Assume Definition~\ref{def:purity-axiom}.  For $A\in\C(X)$ and $\beta\in H^{d-2c}(Z;i^*A)$, define
\[
\Th_i(A)(\beta)\in H_Z^d(X;A)
\]
by multiplying the Thom class with $\beta$, equivalently by the composite
\[
\begin{tikzcd}[column sep=large]
\one_X
\arrow[r, "u(i)"] &
i_*i^!\one_X[2c]
\arrow[r, "{i_*(\id\otimes\beta)}"] &
i_*\bigl(i^!\one_X[2c]\otimes i^*A[d-2c]\bigr)
\arrow[r, "{i_*(\mu_A[d])}"] &
i_*i^!A[d].
\end{tikzcd}
\]
The associated Gysin morphism is
\[
i_*:H^{d-2c}(Z;i^*A)\longrightarrow H^d(X;A),
\qquad
i_*(\beta):=\forg\bigl(\Th_i(A)(\beta)\bigr).
\]
\end{definition}

\begin{theorem}\label{thm:self-int}
Assume Definitions~\ref{def:recollement} and~\ref{def:purity-axiom}.  Then
\[
i^*i_*(\beta)=\beta\smile e(i)
\qquad\text{in }H^d(Z;i^*A)
\]
for every $\beta\in H^{d-2c}(Z;i^*A)$.  If multiplication by $e(i)$ is invertible in the relevant degrees, then
\[
\beta=\frac{i^*i_*(\beta)}{e(i)}.
\]
In particular,
\[
e(i)=i^*i_*(1_Z).
\]
\end{theorem}

\begin{proof}
Pulling the Thom representative of $\beta$ back to $Z$ separates it into the pullback of the Thom class and the factor $\beta$.  By Definition~\ref{def:thom-euler}, the former is $e(i)$; the monoidal compatibility in Definition~\ref{def:purity-axiom} therefore gives
\[
i^*i_*(\beta)=\beta\smile e(i).
\]
The division formula follows under the stated invertibility hypothesis, and the final identity is the case $\beta=1_Z$.
\end{proof}

\subsection{Computation of the canonical localised class}

\begin{definition} \label{def:Thom-iso}
We say that \emph{Thom isomorphism holds for $i$ and $A$} if the map
$
\Th_i(A):H^{d-2c}(Z;i^*A)\to H_Z^d(X;A)
$
of Definition~\ref{def:Th-operator} is an isomorphism.  In the standard oriented purity realisations this is the usual Thom isomorphism obtained from adjunction and purity.
\end{definition}

\begin{theorem} \label{thm:Loc-by-euler}
Assume Definitions~\ref{def:recollement} and~\ref{def:purity-axiom}. Let $A\in\C(X)$ and let $c\in H^d(X;A)$ satisfy $j^*(c)=0$. Assume moreover that Thom isomorphism holds for $i$ and $A$, and that multiplication by $e(i)$ is invertible in the relevant degrees. Then the localisation torsor is a singleton, and its canonical localised class $\Loc_Z(c)$ corresponds, under Thom isomorphism and adjunction, to the class
\begin{equation}\label{eq:gamma-by-denominator}
\gamma=\frac{i^*c}{e(i)}
\qquad\text{in}\qquad
H^{d-2c}(Z;i^*A),
\end{equation}
where the quotient denotes the inverse of multiplication by $e(i)$ in the indicated degrees, and one has
\begin{equation}\label{eq:c-gysin-denominator}
c=i_*\!\left(\frac{i^*c}{e(i)}\right).
\end{equation}
\end{theorem}

\begin{proof}
Since $j^*(c)=0$, the class $c$ admits a supported refinement
$
\widetilde c:\one_X\to i_*i^!A[d].
$
Because Thom isomorphism holds for $i$ and $A$, there exists a unique class
$
\gamma\in H^{d-2c}(Z;i^*A)
$
such that
$
\widetilde c=\Th_i(A)(\gamma).
$
Passing to forget-support gives
$
c=i_*(\gamma).
$
Applying $i^*$ and using Theorem~\ref{thm:self-int}, we obtain
$
i^*c=i^*i_*(\gamma)=\gamma\smile e(i).
$
Since $e(i)$ is invertible, every supported refinement must therefore correspond to the same class
$
\gamma=i^*c/e(i).
$
Thus the supported refinement is unique, proving that the localisation torsor is a singleton. This gives \eqref{eq:gamma-by-denominator}; substituting into $c=i_*(\gamma)$ yields \eqref{eq:c-gysin-denominator}.
\end{proof}
\section{Concentration and localisation of coefficients}\label{sec:concentration}

\begin{definition} \label{def:loc}
Let $\mathscr R=H^*(\pt;\one_{\pt})$ be the graded ground ring of Definition~\ref{def:coh}, and let $S\subset\mathscr R$ be a multiplicative set of homogeneous elements.  For a graded $\mathscr R$-module $M$ write
\[
M[S^{-1}]:=M\otimes_{\mathscr R}S^{-1}\mathscr R.
\]
\end{definition}

\begin{definition} \label{def:conc}
We say \emph{concentration holds for $(i,j)$ on $A$ after inverting $S$} if
$$
H^*(U;j^*A)[S^{-1}]=0.
$$
Equivalently (by the long exact sequence of the localisation triangle), the forget-support map
$$
\forg:\ H_Z^*(X;A)[S^{-1}]\longrightarrow H^*(X;A)[S^{-1}]
$$
is an isomorphism.
\end{definition}

\begin{theorem} \label{thm:universal-euler}
Assume Definitions~\ref{def:recollement} and~\ref{def:purity-axiom}.
Fix a homogeneous multiplicative set $S\subset\mathscr R$ such that concentration holds for $A$ after inverting $S$.
Assume Thom isomorphism holds for $i$ and $A$ after inverting $S$, and assume $e(i)$ is invertible in
$H^{2c}(Z;\one_Z)[S^{-1}]$.
Then for every $\alpha\in H^d(X;A)[S^{-1}]$ one has
$$
\alpha\ =\ i_*\!\left(\frac{i^*\alpha}{e(i)}\right)
\qquad\text{in}\qquad
H^d(X;A)[S^{-1}].
$$
\end{theorem}

\begin{proof}
Concentration implies that $\alpha$ admits a \emph{unique} supported refinement after inverting $S$.
Thom isomorphism writes that supported refinement uniquely as $\Th_i(A)(\gamma)$ for some $\gamma$.
Applying $i^*$ and Theorem~\ref{thm:self-int} gives $i^*\alpha=\gamma\smile e(i)$, hence $\gamma=i^*\alpha/e(i)$ by invertibility.
Finally $\alpha=i_*(\gamma)=i_*(i^*\alpha/e(i))$.
\end{proof}

\section{Equivariant cohomology and the ABBV formula}\label{sec:borel}

\subsection{Multiplicative set and orbit-type annihilation}

Let $T$ be a compact torus acting smoothly on a compact manifold $X$.
Let $Z=X^T$ and $U=X\setminus Z$.
Fix a field $k$ of characteristic $0$ and write
$$
H_T^*(X;k)\ :=\ H^*(ET\times_T X;\,k).
$$
Set
\[
R:=H_T^*(\pt;k)=H^*(BT;k)\cong\mathrm{Sym}_k\bigl(H^2(BT;k)\bigr),
\]
where $H^2(BT;k)$ is placed in cohomological degree $2$.

\begin{definition} \label{def:abbvS}
Let $S\subset R$ be the multiplicative set generated by the nonzero elements of $H^2(BT;k)\subset R^2$.
\end{definition}

\begin{lemma} \label{lem:kill-isotropy}
If $H\subsetneq T$ is a proper closed subgroup, then
$
H_T^*(T/H;k)[S^{-1}]=0.
$
\end{lemma}

\begin{proof}
There is a canonical identification
$$
ET\times_T (T/H)\ \simeq\ EH/H\ \simeq\ BH,
$$
hence $H_T^*(T/H;k)\cong H^*(BH;k)$.
Since $k$ has characteristic $0$, the finite component group of $H$ contributes no positive-degree cohomology, and
\[
H^*(BH;k)\cong H^*(B(H^\circ);k)\cong \mathrm{Sym}_k\bigl(H^2(BH;k)\bigr).
\]
The restriction map $H^2(BT;k)\to H^2(BH;k)$ has nonzero kernel because $H\subsetneq T$ has strictly smaller identity-component rank.  Choose a nonzero
$\ell\in H^2(BT;k)$ which restricts to zero on $BH$.  Then $\ell\in S$ acts trivially on $H^*(BH;k)$, so localising at $S$ annihilates the module.  Hence $H_T^*(T/H;k)[S^{-1}]=0$.
\end{proof}

\subsection{Borel concentration}

The following theorem is due to Illman~\cite{Illman}.
\begin{theorem}\label{thm:illman-cw}
A smooth proper action of a compact Lie group on a compact smooth manifold admits a finite $G$-CW structure
compatible with orbit types; the fixed locus is a subcomplex.
\end{theorem}

\begin{theorem}\label{thm:borel-conc}
With $S$ as in Definition~\ref{def:abbvS}, one has
\[
H_T^*(U;k)[S^{-1}]=0,
\]
and restriction to the fixed locus is an isomorphism
\[
H_T^*(X;k)[S^{-1}]\xrightarrow{\ \sim\ }H_T^*(Z;k)[S^{-1}].
\]
\end{theorem}

\begin{proof}
Choose an Illman finite $T$-CW structure on $X$ compatible with orbit types and with $Z$ as a subcomplex.  The complement $U$ has finite $T$-CW homotopy type with only proper isotropy groups.  Lemma~\ref{lem:kill-isotropy}, applied cell by cell, therefore gives
\[
H_T^*(U;k)[S^{-1}]=0.
\]
The same cellular argument for the relative $T$-CW pair $(X,Z)$ gives
\[
H_T^*(X,Z;k)[S^{-1}]=0.
\]
The long exact sequence of the pair then shows that $H_T^*(X;k)[S^{-1}]\to H_T^*(Z;k)[S^{-1}]$ is an isomorphism.
\end{proof}

\subsection{Invertibility of Euler classes and ABBV}

For the Gysin maps below, fix compatible orientations of the normal bundles of the connected fixed components over $k$; this is automatic, for example, in the presence of a $T$-invariant almost-complex structure.  Equivalently, one may choose $T$-invariant complex structures on the non-trivial real isotypical summands used in the proof below and use their induced orientations.

\begin{lemma}\label{lem:unit-nilp}
If $A$ is a ring, $u\in A^\times$, and $n\in A$ is nilpotent, then $u(1+n)$ is invertible.
\end{lemma}

\begin{proof}
If $n^N=0$, then $(1+n)^{-1}=\sum_{m=0}^{N-1}(-n)^m$.
\end{proof}

\begin{lemma}\label{lem:euler-invert}
Let $F$ be a connected component of $Z=X^T$ and let $N_{F/X}$ be the $T$-equivariant normal bundle.
Then $e_T(N_{F/X})$ becomes invertible in $H_T^*(F;k)[S^{-1}]$.
\end{lemma}

\begin{proof}
There is a canonical identification
$$
H_T^*(F;k)\ \cong\ H^*(F;k)\otimes_k R.
$$
Because $F$ is compact, every element of positive cohomological degree in $H^*(F;k)$ is nilpotent, hence the ideal
$H^{>0}(F;k)\otimes_k R \subset H_T^*(F;k)$ is nilpotent.

Since $F$ is fixed, the normal $T$-representation has no trivial weight.  Decompose $N_{F/X}$ into its non-trivial real isotypical summands and choose a $T$-invariant complex structure on each of them; this fixes the corresponding equivariant orientation.  After applying the splitting principle, the resulting equivariant complex bundle splits as a direct sum of $T$-equivariant complex line bundles $L_{\chi_j}$ with nonzero weights $\chi_j\in H^2(BT;k)\setminus\{0\}$.
Under the splitting principle one has
$$
c_1^T(L_{\chi_j})=\chi_j+c_1(L_{\chi_j}),
$$
with $c_1(L_{\chi_j})$ of positive ordinary cohomological degree.  After inverting $S$, each nonzero $\chi_j$ is a unit, and therefore
$$
e_T(N_{F/X})
=\prod_j\bigl(\chi_j+c_1(L_{\chi_j})\bigr)
=\Big(\prod_j\chi_j\Big)
\prod_j\left(1+\frac{c_1(L_{\chi_j})}{\chi_j}\right).
$$
The first factor is a unit and the second is of the form $1+n$ with $n$ in the nilpotent ideal generated by $H^{>0}(F;k)$.  Lemma~\ref{lem:unit-nilp} now gives the result.
\end{proof}

\begin{corollary}\label{cor:abbv}
For $\alpha\in H_T^*(X;k)[S^{-1}]$ one has
$$
\alpha\ =\ \sum_{F\subset X^T} i_{F*}\!\left(\frac{i_F^*\alpha}{e_T(N_{F/X})}\right)
\qquad\text{in}\qquad
H_T^*(X;k)[S^{-1}],
$$
where the sum runs over connected components $F$ of $X^T$.
\end{corollary}

\begin{proof}
Apply Theorem~\ref{thm:universal-euler} in the Borel model, using Theorem~\ref{thm:borel-conc}
and Lemma~\ref{lem:euler-invert}.
\end{proof}

\section{Multiplicative analogue: equivariant \texorpdfstring{$K$}{K}-theory}\label{sec:K}

\begin{remark}
The relation with equivariant algebraic $K$-theory is multiplicative rather than a direct instance of the Hom-based cohomology of Section~\ref{sec:ambient}. Indeed, in $\Perf^T(X)$ the unit object is $\OO_X$ and one has
\begin{equation*}
\End_{\Perf^T(X)}(\OO_X)\cong H^0(X,\OO_X)^T,
\end{equation*}
rather than a Grothendieck group. The relevant localisation phenomenon is governed by the multiplicative denominator $\lambda_{-1}(N^\vee)$, the counterpart of the Euler class in the preceding cohomological discussion.
\end{remark}

\begin{proposition} \label{prop:K-euler}
Let $i:Z\hookrightarrow X$ be a regular immersion of codimension $c$ of $T$-schemes, with conormal bundle
$N_{Z/X}^\vee$.
Then in $K_0^T(Z)$ one has
$$
i^*i_*(1_Z)\ =\ \lambda_{-1}(N_{Z/X}^\vee)
\ :=\ \sum_{k=0}^c (-1)^k[\Lambda^k N_{Z/X}^\vee].
$$
\end{proposition}

\begin{proof}
In equivariant $K$-theory, $i_*(1_Z)=[\OO_Z]\in K_0^T(X)$.
Pulling back, $i^*i_*(1_Z)$ is the class of the derived tensor product
$\OO_Z\otimes_{\OO_X}^{\mathbf L}\OO_Z$ in $K_0^T(Z)$, hence
$$
i^*i_*(1_Z)\ =\ \sum_{k\ge 0} (-1)^k[\mathrm{Tor}_k^{\OO_X}(\OO_Z,\OO_Z)].
$$
For a regular immersion, the standard identification of Tor-sheaves gives
$$
\mathrm{Tor}_k^{\OO_X}(\OO_Z,\OO_Z)\ \cong\ \Lambda^k N_{Z/X}^\vee,
\qquad 0\le k\le c,
$$
and $\mathrm{Tor}_k=0$ for $k>c$. Substituting yields the formula.
\end{proof}

The following is the standard torus consequence of Thomason's concentration theorem~\cite{Thomason}.
\begin{theorem}\label{thm:thomason-localisation}
Let $T$ be an algebraic torus acting on a quasi-projective scheme $X$, and let $i:X^T\hookrightarrow X$ be the fixed-locus immersion.  Thomason's concentration theorem identifies the localised equivariant $G$-theory of $X$ with that of $X^T$ via the proper pushforward $i_*$, after localising the representation ring at the multiplicative set generated by $1-\chi$ for nontrivial characters $\chi$.  If $X$ is smooth, then $X^T$ is smooth, $i$ is regular, and
\[
i^*i_*(\alpha)=\lambda_{-1}(N^\vee)\cdot\alpha.
\]
The class $\lambda_{-1}(N^\vee)$ is invertible after the same localisation, so the inverse to $i_*$ is given by restriction followed by division by $\lambda_{-1}(N^\vee)$.  This yields the usual equivariant $K$-theoretic localisation and Lefschetz--Riemann--Roch formulae.
\end{theorem}

\section{Lefschetz-type decompositions and standard realisations}\label{sec:lefschetz}

\subsection{Lefschetz objects and supported classes}

\begin{definition} \label{def:Lef}
Let $f:X\to X$ be a morphism such that the relevant shriek functors exist. Let $\Delta:X\hookrightarrow X\times X$ be the diagonal and $\Gamma_f:X\hookrightarrow X\times X$ the graph. The associated \emph{Lefschetz object} is
\begin{equation*}
L_f:=\Delta^!\big((\Gamma_f)_!\one_X\big)\in \C(X).
\end{equation*}
A \emph{supported Lefschetz class} for $f$ consists of a coefficient object $A\in\C(\pt)$ together with a class
\begin{equation*}
\lambda_f\in H^d(X;p_X^*A)
\end{equation*}
whose restriction to the open complement of the fixed-point locus $S=\Fix(f)$ vanishes; equivalently, its support is contained in $S$.
\end{definition}

\begin{remark}
The graph--diagonal construction yields the object $L_f$. In concrete fixed-point theories one extracts from it a cohomology class $\lambda_f$ with coefficients pulled back from the point. The localisation theory applies to this class, and the statements below require only its existence and support on the fixed locus.
\end{remark}

\begin{theorem}\label{thm:Lefschetz-formal}
Let $f:X\to X$ be a morphism, let $S=\Fix(f)$ with inclusion $i:S\hookrightarrow X$ and open complement $j:X\setminus S\hookrightarrow X$, and let
\[
\lambda_f\in H^d(X;p_X^*A)
\]
be a supported Lefschetz class, so that $j^*(\lambda_f)=0$.  Then:
\begin{enumerate}
\item there is a localisation torsor
\[
\Loc_S^{\mathrm{tor}}(\lambda_f)
\subset
\Hom_{\C(S)}(\one_S,i^!p_X^*A[d]);
\]
\item for every $\ell\in\Loc_S^{\mathrm{tor}}(\lambda_f)$ one has
\[
\int_X\lambda_f=\int_S\ell.
\]
If $S=\bigsqcup_\alpha S_\alpha$ is a finite disjoint union and $\ell=(\ell_\alpha)_\alpha$ is the corresponding decomposition, then
\[
\int_X\lambda_f=\sum_\alpha\int_{S_\alpha}\ell_\alpha;
\]
\item if an oriented purity structure and a Thom isomorphism are available for $i$, and multiplication by the Euler class $e(i)$ is an isomorphism after the relevant coefficient localisation or periodic/twisted passage, then the torsor is a singleton and
\[
\Loc_S(\lambda_f)=\frac{i^*\lambda_f}{e(i)}.
\]
\end{enumerate}
\end{theorem}

\begin{proof}
The first statement is Theorem~\ref{thm:Loc-factor}; the second is Theorem~\ref{thm:global-local}; and the third follows from Theorem~\ref{thm:Loc-by-euler}.
\end{proof}
\subsection{Equivariant motivic fixed-point localisation}

Let $G$ be a linearly reductive linear algebraic group over a base scheme $S$ in the range of Hoyois's equivariant six-operation theory (in particular, under the relevant tameness, resolution and equivariant quasi-projectivity hypotheses), and let $\C(-)$ be an equivariant motivic coefficient theory in which the functorialities used above are available.  In this setting the six operations, gluing and purity are provided by Hoyois \cite{HoyoisSix}, while fundamental classes, Gysin maps and Euler classes in the motivic setting are developed by D\'eglise--Jin--Khan \cite{DJK}. Let $X$ be a smooth proper $G$-scheme over $S$, let $f:X\to X$ be a $G$-equivariant endomorphism, and let
\[
i:F=\Fix(f)\hookrightarrow X,
\qquad
j:U=X\setminus F\hookrightarrow X
\]
be the fixed-point immersion and its open complement. In Hoyois' quadratic refinement of the Grothendieck--Lefschetz--Verdier trace formula, the global trace is expressed through the fixed-point scheme \cite[Theorem~1.3]{HoyoisTrace}; motivated by that result, we assume that there exists a supported Lefschetz class
\begin{equation*}
\lambda_f\in H^0(X;p_X^*A)
\end{equation*}
for $f$, supported on $F$, equivalently satisfying $j^*(\lambda_f)=0$. Under this support statement, the conclusions below are formal consequences of the general localisation results proved in Sections~\ref{sec:UnivLoc} and~\ref{sec:lefschetz}.

\begin{corollary}\label{cor:motivic-fixed-point}
Under the preceding assumptions, $\lambda_f$ determines the localisation torsor
\[
\Loc_F^{\mathrm{tor}}(\lambda_f)
\subset
\Hom_{\C(F)}(\one_F,i^!p_X^*A).
\]
If $F=\bigsqcup_\alpha F_\alpha$ and $\ell=(\ell_\alpha)_\alpha$ is any point of this torsor written according to the connected components, then
\[
\int_X\lambda_f
=
\sum_\alpha\int_{F_\alpha}\ell_\alpha.
\]
If, moreover, $i$ is regular and the purity and concentration hypotheses of the Euler-denominator construction are satisfied, the torsor is a singleton and its unique element is the corresponding Euler-divided class.
\end{corollary}

\begin{proof}
Apply Theorems~\ref{thm:Loc-factor},~\ref{thm:global-local} and~\ref{thm:universal-euler}.
\end{proof}

\begin{remark}\label{rem:motivic-fixed-point-comparison}
If the supported class $\lambda_f$ is chosen, under the relevant realisation, to be the class entering the quadratic motivic fixed-point formula of Hoyois \cite{HoyoisTrace}, then the rigidified local term above agrees with the corresponding quadratic local contribution. In particular, when the base is a field and the fixed points are isolated and \'{e}tale, one recovers the associated Grothendieck--Witt-valued local terms. For rigidified localisation statements in quadratic and, more generally, $SL_{\eta}$-oriented theories, compare also Levine \cite{LevineWitt} and D'Angelo \cite{DAngeloSLeta}.
\end{remark}

\subsection{External-product compatibility in geometric realisations}\label{subsec:Loc-box-model}

The product compatibility used below is a specific form of Proposition~\ref{prop:Loc-box}, now stated under the additional hypotheses required to relate $\boxtimes$ to extraordinary pullback and closed pushforward.

\begin{proposition} \label{prop:Loc-box-model}
Let $i:Z\hookrightarrow X$ and $i':Z'\hookrightarrow X'$ be closed immersions with open complements
$j:U\hookrightarrow X$ and $j':U'\hookrightarrow X'$. Let $A\in\C(X)$, $A'\in\C(X')$, and let
$c\in H^d(X;A)$, $c'\in H^{d'}(X';A')$ satisfy $j^*(c)=0$ and ${j'}^*(c')=0$.

Assume that the bifunctor
$
\boxtimes:\C(X)\times\C(X')\to\C(X\times X')
$
exists, is biexact, and is compatible with pullback and proper pushforward. Assume moreover that for the product immersion $i\times i'$ one has the corresponding compatibilities of $\boxtimes$ with closed pushforward and extraordinary pullback, functorially and compatibly with counits, so that
\begin{equation*}
((i_*C)\boxtimes (i'_*C'))\simeq (i\times i')_*(C\boxtimes C')
\end{equation*}
and
\begin{equation*}
(i\times i')^!(A\boxtimes A')\simeq i^!A\boxtimes {i'}^!A'.
\end{equation*}
Then:
\begin{itemize}
\item for every pair of supported refinements
$
\widetilde c\in \SLoc_Z^d(c),  \ 
\widetilde c'\in \SLoc_{Z'}^{d'}(c'),
$
there is a naturally associated supported refinement
\begin{equation}\label{eq:product-supported-refinement}
\widetilde c\boxtimes \widetilde c'\in \SLoc_{Z\times Z'}^{d+d'}(c\boxtimes c');
\end{equation}
\item consequently one obtains a natural map of torsors
\begin{equation}\label{eq:Loc-box-model-torsor}
\Loc_Z^{\mathrm{tor}}(c)\times \Loc_{Z'}^{\mathrm{tor}}(c')
\longrightarrow
\Loc^{\mathrm{tor}}_{Z\times Z'}(c\boxtimes c');
\end{equation}
\item if the three localisation torsors in question are singletons, then
\begin{equation}\label{eq:Loc-box-model}
\Loc_{Z\times Z'}(c\boxtimes c')
=
\Loc_Z(c)\boxtimes \Loc_{Z'}(c').
\end{equation}
\end{itemize}
\end{proposition}

\begin{proof}
Let
$$
\widetilde c:\one_X\to i_*i^!A[d],
\qquad
\widetilde c':\one_{X'}\to i'_*i'^!A'[d']
$$
be supported refinements of $c$ and $c'$, so that
$$
\epsilon_{A[d]}\circ \widetilde c=c,
\qquad
\epsilon_{A'[d']}\circ \widetilde c'=c'.
$$
Taking external products gives
\begin{align*}
\one_{X\times X'}
&\simeq \one_X\boxtimes \one_{X'}
\xrightarrow{\widetilde c\boxtimes \widetilde c'}
(i_*i^!A[d])\boxtimes (i'_*i'^!A'[d']) \\
&\simeq (i\times i')_*(i^!A\boxtimes i'^!A')[d+d']
\simeq (i\times i')_*(i\times i')^!(A\boxtimes A')[d+d'].
\end{align*}
By compatibility with the counits, the composite of this morphism with
$$
(i\times i')_*(i\times i')^!(A\boxtimes A')[d+d']\longrightarrow (A\boxtimes A')[d+d']
$$
is precisely $c\boxtimes c'$. Hence \eqref{eq:product-supported-refinement} defines a supported refinement of $c\boxtimes c'$ along $Z\times Z'$. This proves the first assertion.

Passing to adjoints under $(i\times i')_*\dashv (i\times i')^!$ yields the natural map of torsors \eqref{eq:Loc-box-model-torsor}. If the two source torsors and the target torsor are singletons, the unique pair maps to the unique target element, giving \eqref{eq:Loc-box-model}.
\end{proof}


\subsection{Standard six-functor realisations}

The construction applies in the standard sheaf-theoretic settings with recollement, exceptional pullback, Cartesian base change, proper direct image and external products. Only the identifications needed for the torsor are used below.

For a complex algebraic or analytic space one may take $\C(X)=D^b_c(X;k)$, with unit $k_X$ and
\[
H_Z^d(X;A)=\Hom_{D^b_c(X;k)}(k_X,i_*i^!A[d])
\cong
\Hom_{D^b_c(Z;k)}(k_Z,i^!A[d]);
\]
see \cite{KS,BBD}. For schemes, and for Deligne--Mumford or Artin stacks in the usual constructible range, the corresponding statements hold for $\ell$-adic complexes by the six-operation formalisms of \cite{SGA4,BBD,LO1,LO2}. Hence the torsors and their excision, base-change and external-product functorialities specialise directly to these settings.  Proper pushforward and the local-to-global index statements additionally require the trace or orientation data specified in Proposition~\ref{prop:Loc-proper} and Section~\ref{sec:dual}.

When a regular immersion of codimension $c$ satisfies the relevant purity isomorphism, the supported group is identified with the shifted cohomology of the closed stratum. In the complex topological case this has the form
\[
H_Z^d(X;k_X)\cong H^{d-2c}(Z;k_Z),
\]
while in the $\ell$-adic setting the usual Tate twist is present. Under these identifications the forget-support morphism becomes the Gysin map.  Pulling it back to $Z$ gives the self-intersection operator, namely multiplication by the Euler class.  Proposition~\ref{prop:ambiguity-link-transgression} identifies the secondary group with the global-to-link part of the corresponding Euler kernel, and Proposition~\ref{prop:canonicity-before-euler-inversion} shows that injectivity of Euler multiplication forces uniqueness.  After concentration and Euler inversion the canonical local term is the usual Euler-divided class.

For smooth Deligne--Mumford stacks the same argument applies to characteristic classes and fixed loci whenever the chosen equivariant or $\ell$-adic theory provides purity and concentration. The supported torsor is intrinsic, whereas division by the Euler class occurs only after rigidification.

\section{Milnor localisation torsors}\label{sec:milnor-localization-torsors}

Secondary localisation applies naturally to characteristic-class defects on singular spaces.  It records the supported representatives of such a defect before singularity-theoretic data select a distinguished one.  The local geometry is closely related to the classical theory of indices on singular varieties, including the Schwartz index, the local Euler obstruction and the GSV-index; see Brasselet--Seade--Suwa and Seade's survey \cite{BrasseletSeadeSuwaVectorFields,BrasseletSeadeSuwaOneForms,SeadeOverview}.

The characteristic classes below are often homological.  For the Chern--Schwartz--MacPherson, Fulton and Hirzebruch classes we therefore fix a six-functor realisation in which they are represented in the relevant Hom group, for example by means of a dualising object and the orientation data of Section~\ref{sec:dual}.  Motivic Chern classes are treated in the corresponding $K$-theoretic analogue.  Thus the notation $\SLoc_Z$ is used only after the relevant realisation has been fixed.

Let $X$ be singular, and write
\[
i:Z=\operatorname{Sing}(X)\hookrightarrow X,
\qquad
j:U=X_{\mathrm{reg}}\hookrightarrow X
\]
for the singular locus and the regular open complement.  Let $C$ be a characteristic theory with values in a coefficient object $A$ in the sense used here.  Assume that $C$ provides a functorial class $C(X)$ and a virtual, expected, or Fulton-type class $C^{\mathrm{vir}}(X)$, and that these two classes agree on the regular locus:
\[
j^*C^{\mathrm{vir}}(X)=j^*C(X).
\]
Set
\[
\Delta_C(X):=C^{\mathrm{vir}}(X)-C(X).
\]
Then $j^*\Delta_C(X)=0$.

\begin{definition}\label{def:milnor-localization-torsor}
The \emph{Milnor localisation torsor} of $X$ with respect to $C$ is
\[
\mathfrak{Mil}^C_Z(X)
:=
\SLoc_Z(\Delta_C(X)).
\]
Equivalently,
\[
\mathfrak{Mil}^C_Z(X)
=
\{\widetilde\Delta\in H^*_Z(X;A)
\mid
\forg(\widetilde\Delta)=\Delta_C(X)\}.
\]
Its elements are called \emph{supported Milnor refinements} of the characteristic-class defect $\Delta_C(X)$.
\end{definition}

\begin{theorem}\label{thm:milnor-ambiguity-torsor}
The Milnor localisation torsor $\mathfrak{Mil}^C_Z(X)$ is nonempty and is a torsor under the secondary transgression group
\[
\TLoc^C_Z(X)
:=
\operatorname{im}\bigl(
\delta:H^{*-1}(U;j^*A)\to H^*_Z(X;A)
\bigr).
\]
Equivalently, if
\[
\widetilde\Delta_1,\widetilde\Delta_2
\in
\mathfrak{Mil}^C_Z(X),
\]
then their difference is a unique element of $\TLoc^C_Z(X)$.
\end{theorem}

\begin{proof}
By construction,
\[
j^*\Delta_C(X)=j^*C^{\mathrm{vir}}(X)-j^*C(X)=0.
\]
The assertion is the localisation torsor theorem, Theorem~\ref{thm:Loc-factor}, applied to $c=\Delta_C(X)$ and $Z=\operatorname{Sing}(X)$.
\end{proof}

\begin{corollary}\label{cor:milnor-canonical}
If
$
\TLoc^C_Z(X)=0,
$
then the Milnor defect $\Delta_C(X)$ has a canonical supported refinement on $Z$.  Equivalently, the corresponding local Milnor class is formally unique.
Conversely, when $\TLoc^C_Z(X)\neq 0$, the localisation triangle alone does not determine a distinguished local Milnor class; any canonical choice must come from additional geometric structure.
\end{corollary}

Thus $j^*\Delta_C(X)=0$ gives support on $Z$, whereas $\TLoc^C_Z(X)=0$ gives uniqueness of the supported representative.  The latter is the secondary canonicity condition.

\begin{theorem}\label{thm:milnor-characterisation}
Let $\mathcal M^C_Z$ be any assignment which associates to the Milnor defect $\Delta_C(X)$ a class
\[
\mathcal M^C_Z(X)\in \Hom_{\C(Z)}(\one_Z,i^!A[*])
\]
whose adjoint
\[
\widetilde{\mathcal M}^C_Z(X):\one_X\to i_*i^!A[*]
\]
is compatible with the forget-support morphism, namely
$
\forg\bigl(\widetilde{\mathcal M}^C_Z(X)\bigr)=\Delta_C(X)$.
Then $\mathcal M^C_Z(X)$ determines a point of the Milnor localisation torsor
\[
\mathfrak{Mil}^C_Z(X).
\]
If $\mathfrak{Mil}^C_Z(X)$ is a singleton, every such compatible construction gives the same localised Milnor class.
\end{theorem}

\begin{proof}
The compatibility condition says precisely that the adjoint class $\widetilde{\mathcal M}^C_Z(X)$ is a supported refinement of $\Delta_C(X)$.  Hence it lies in $\SLoc_Z(\Delta_C(X))=\mathfrak{Mil}^C_Z(X)$.  The final assertion follows from uniqueness when the torsor is a singleton.  Equivalently, this is Proposition~\ref{thm:Loc-universal} specialised to the class $c=\Delta_C(X)$.
\end{proof}

\begin{proposition}\label{prop:milnor-comparison}
Let
$
\Phi:C_1\longrightarrow C_2
$
be a natural transformation of characteristic theories which, in the chosen realisations, is compatible with the open--closed localisation sequences and with the functorial and virtual classes, so that
\[
\Phi\bigl(C^{\mathrm{vir}}_1(X)\bigr)=C^{\mathrm{vir}}_2(X),
\qquad
\Phi\bigl(C_1(X)\bigr)=C_2(X).
\]
Then
\[
\Phi\bigl(\Delta_{C_1}(X)\bigr)=\Delta_{C_2}(X),
\]
and $\Phi$ induces a morphism of Milnor localisation torsors
\[
\Phi_*:
\mathfrak{Mil}^{C_1}_Z(X)
\longrightarrow
\mathfrak{Mil}^{C_2}_Z(X).
\]
Moreover, this morphism is compatible with the corresponding secondary transgression groups
\[
\TLoc^{C_1}_Z(X)\longrightarrow \TLoc^{C_2}_Z(X).
\]
\end{proposition}

\begin{proof}
The first identity follows by subtracting the two compatibility identities for $\Phi$.  Functoriality of the localisation long exact sequence, and hence of the boundary maps and supported fibres, then carries supported refinements of $\Delta_{C_1}(X)$ to supported refinements of $\Delta_{C_2}(X)$.  This gives the asserted morphism of torsors and its compatibility with the images of the corresponding boundary maps.
\end{proof}

Natural transformations between generalised Milnor theories induce comparison morphisms between the corresponding supported-refinement torsors.

\begin{corollary}\label{cor:milnor-local-index}
Assume that the coefficient theory is endowed with the duality, proper-pushforward and orientation data used in Section~\ref{sec:dual}.  Let
\[
\widetilde\Delta\in \mathfrak{Mil}^C_Z(X)
\]
be a supported Milnor refinement, and let $\ell_{\widetilde\Delta}$ be its adjoint point of $\Loc_Z^{\mathrm{tor}}(\Delta_C(X))$.  Whenever Definition~\ref{def:local-index} applies,
\[
\int_X \Delta_C(X)=\int_Z \ell_{\widetilde\Delta}.
\]
If $Z=\coprod_\alpha Z_\alpha$ is a finite disjoint decomposition and $\widetilde\Delta_\alpha$ denotes the component of $\widetilde\Delta$ supported on $Z_\alpha$, then
\[
\int_X \Delta_C(X)
=
\sum_\alpha \int_{Z_\alpha}\widetilde\Delta_\alpha.
\]
\end{corollary}

\begin{proof}
This is Theorem~\ref{thm:global-local} applied to the class $c=\Delta_C(X)$, together with the additivity over disjoint decompositions stated there.
\end{proof}

A classical local Milnor class therefore gives a geometrically distinguished point of $\mathfrak{Mil}^C_Z(X)$; the following subsections make this distinction explicit.

\subsection{Isolated singularities, links, and Milnor refinements}
\label{subsec:isolated-singularities-milnor-refinements}

Consider now a hypersurface with isolated singularities.  Let $M$ be a smooth complex algebraic variety and let
\[
X=f^{-1}(0)\subset M
\]
be a hypersurface with isolated singular locus
\[
Z=\operatorname{Sing}(X)=\{p_1,\ldots,p_s\}.
\]
Put $U=X_{\mathrm{reg}}$, and let
\[
j:U\hookrightarrow X,
\qquad
 i:Z\hookrightarrow X
\]
denote the complementary open and closed immersions.

Let $C$ be as above, with defect $\Delta_C(X)=C^{\mathrm{vir}}(X)-C(X)$.  Since $j^*\Delta_C(X)=0$, one has
\[
\mathfrak{Mil}^C_Z(X)=\SLoc_Z(\Delta_C(X)).
\]

Since $Z$ is finite, a supported refinement of $\Delta_C(X)$ may be regarded as
a choice of local supported contributions
\[
\widetilde{\Delta}
=
\sum_{\alpha=1}^s \widetilde{\Delta}_{p_\alpha},
\qquad
\widetilde{\Delta}_{p_\alpha}\in H^*_{\{p_\alpha\}}(X;A),
\]
whose image under the forget-support morphism is the global defect $\Delta_C(X)$.  The torsor therefore records the possible local decompositions before a distinguished singularity-theoretic construction selects one.

By Theorem~\ref{thm:milnor-ambiguity-torsor}, the secondary indeterminacy of such local
realisations is
\[
\TLoc^C_Z(X)
=
\operatorname{im}\left(
\delta:H^{*-1}(U;j^*A)\longrightarrow H^*_Z(X;A)
\right).
\]
Thus the global defect has a unique supported decomposition precisely when this secondary group vanishes; otherwise the regular locus contributes a non-trivial transgression.

For each singular point $p_\alpha$, choose a sufficiently small Milnor ball
$B_\alpha\subset M$, and write
\[
L_\alpha=X\cap \partial B_\alpha
\]
for the link of the singularity.  By excision, the local part of the
localisation sequence at $p_\alpha$ is computed by the pair
\[
(X\cap B_\alpha,\,(X\cap B_\alpha)\setminus\{p_\alpha\}).
\]
The punctured neighbourhood
$
(X\cap B_\alpha)\setminus\{p_\alpha\}
$
deformation retracts onto $L_\alpha$.  Hence the local component of the
boundary morphism
\[
\delta:H^{*-1}(U;j^*A)\longrightarrow H^*_Z(X;A)
\]
is governed, after excision, by the cohomology of the links $L_\alpha$.
Whenever the local closed immersions admit the purity, orientation and Thom-identification hypotheses of
Proposition~\ref{prop:ambiguity-link-transgression}, the secondary group is the image of the global-to-link transgression.  More precisely, after the Thom identifications it is obtained from the composite
\[
H^{*-1}(U;j^*A)\longrightarrow
\bigoplus_\alpha H^{*-1}(L_\alpha;A|_{L_\alpha})\longrightarrow
\bigoplus_\alpha H^{*-2c_\alpha}(p_\alpha;i^*A),
\]
where the first arrow is restriction and the second is the sum of the local Gysin maps.  Thus the Milnor localisation torsor records not only that $\Delta_C(X)$ is supported on the singular set, but also how the topology of the regular locus constrains its local supported representatives.

\subsection{Vanishing cycles as a rigidification}
\label{subsec:vanishing-cycles-rigidification}

In the hypersurface case, nearby- and vanishing-cycle constructions supply the additional singularity-theoretic data which can select a distinguished point of
\[
\mathfrak{Mil}^C_Z(X)=\SLoc_Z(\Delta_C(X)).
\]
The torsor records the ambiguity allowed by localisation, while vanishing cycles provide the geometric rigidification and select a distinguished supported refinement.
Under the duality and orientation hypotheses of Section~\ref{sec:dual}, any
supported Milnor refinement
$
\widetilde{\Delta}\in\mathfrak{Mil}^C_Z(X)
$
defines local indices at the singular points.  If
\[
\widetilde{\Delta}
=
\sum_{\alpha=1}^s\widetilde{\Delta}_{p_\alpha},
\]
then Corollary~\ref{cor:milnor-local-index} gives a decomposition
\[
\int_X\Delta_C(X)
=
\sum_{\alpha=1}^s
\int_{p_\alpha}\widetilde{\Delta}_{p_\alpha}.
\]
When the torsor is rigidified by a geometric construction, this becomes the
usual decomposition of a global Milnor-type invariant into local singular
contributions.  When the secondary transgression group is non-zero, the same formula holds
for each supported refinement, and the torsor records the possible variation of
the local decomposition.

The construction gives the following instances.
\begin{itemize}
\item For the Chern--Schwartz--MacPherson/Fulton defect,
\[
\mathfrak{Mil}^{c}_Z(X)
=
\SLoc_Z\bigl(c_F(X)-c_{\mathrm{SM}}(X)\bigr).
\]

\item For motivic Chern classes,
\[
\mathfrak{Mil}^{mC_y}_Z(X)
=
\SLoc_Z\bigl(mC_y^{\mathrm{vir}}(X)-mC_y(X)\bigr).
\]

\item For Hirzebruch classes,
\[
\mathfrak{Mil}^{T_y}_Z(X)
=
\SLoc_Z\bigl(T^{\mathrm{vir}}_{y*}(X)-T_{y*}(X)\bigr).
\]
\end{itemize}
These virtual-minus-functorial defects and their local forms are classical; see \cite{ParusinskiPragaczJAG,YokuraMotivicMilnor,BrasseletSchuermannYokura,CappellMaximSchuermannShaneson,MaximSaitoYang,CallejasMorgadoSeadeChernClasses} and, for the index-theoretic viewpoint, \cite{BrasseletSeadeSuwaVectorFields,BrasseletSeadeSuwaOneForms}.

\medskip
\noindent\textbf{Nodal and cuspidal plane cubics.}
Consider the two plane cubics
$
X_{\mathrm n}=\{y^2z=x^2(x+z)\},
$ and $
X_{\mathrm c}=\{y^2z=x^3\}
\subset \mathbb P^2.
$
Each has a unique singular point $p=[0:0:1]$; it is a node for $X_{\mathrm n}$ and a cusp for $X_{\mathrm c}$.  Both are local complete intersections of degree three, and their virtual tangent bundles have the same $K$-theory expression as for a smooth plane cubic.  In particular their virtual $\chi_y$-genus is zero.

For either cubic, let $\nu:\mathbb P^1\to X$ denote the normalisation.  Over the node the fibre of $\nu$ consists of two points, whereas over the cusp it consists of one point.  Additivity in the Grothendieck group of varieties therefore gives
$
[X_{\mathrm n}]=[\mathbb P^1]-[\mathrm{pt}],
$ and $
[X_{\mathrm c}]=[\mathbb P^1].
$
Since $\chi_y(\mathbb P^1)=1-y$, one obtains
\[
\chi_y(X_{\mathrm n})=-y,
\qquad
\chi_y(X_{\mathrm c})=1-y.
\]
Hence the virtual-minus-functorial defects have respective $\chi_y$-values
$
y$ and $
y-1.$
The defects are supported at $p$.  Since pushforward to a point sends $[\mathcal O_p]$ and $[p]$ to $1$, it is injective on the cyclic subgroups generated by these point classes.  Thus the numerical defects determine the corresponding point-supported classes.  Vanishing cycles select the same distinguished points of the localisation torsors, and in the standard $G_0$- and Borel--Moore/Chow realisations these representatives are
\[
\widetilde\Delta^{mC_y}_{p}(X_{\mathrm n})
=y[\mathcal O_p],
\qquad
\widetilde\Delta^{mC_y}_{p}(X_{\mathrm c})
=(y-1)[\mathcal O_p],
\]
and, after the normalised motivic Riemann--Roch transformation $td_{(1+y)*}$,
\[
\widetilde\Delta^{T_y}_{p}(X_{\mathrm n})
=y[p],
\qquad
\widetilde\Delta^{T_y}_{p}(X_{\mathrm c})
=(y-1)[p].
\]
Because the support is zero-dimensional, the Riemann--Roch transformation does not alter these coefficients.  At $y=-1$ the Hirzebruch defects specialise to the Chern defects,
\[
\widetilde\Delta^{c}_{p}(X_{\mathrm n})=-[p],
\qquad
\widetilde\Delta^{c}_{p}(X_{\mathrm c})=-2[p],
\]
in agreement with the formula $(-1)^{\dim X}\mu_p[p]$, since the Milnor numbers of the node and the cusp are $1$ and $2$, respectively.  The $y$-refinement retains more information: at $y=0$ the nodal defect vanishes, while the cuspidal defect is $-[p]$.  Thus the Chern, motivic Chern and Hirzebruch defects are not three unrelated examples.  In this elementary case they are linked by specialisation and Riemann--Roch, while the polynomial parameter records information which disappears in the Chern specialisation; compare \cite{BrasseletSchuermannYokura,CappellMaximSchuermannShaneson}.

\medskip
\noindent\textbf{A higher-dimensional singular stratum.}
The preceding computation admits a global higher-dimensional extension in which the singular support is no longer zero-dimensional.  Let $r\geq 1$, put $Y=\mathbb P^r$, and set
\[
\mathcal X_{\mathrm n}=X_{\mathrm n}\times Y,
\qquad
\mathcal X_{\mathrm c}=X_{\mathrm c}\times Y
\subset M:=\mathbb P^2\times Y.
\]
Both are hypersurfaces of bidegree $(3,0)$ in the smooth variety $M$, of dimension $r+1$, and
\[
\operatorname{Sing}(\mathcal X_{\mathrm n})
=
\operatorname{Sing}(\mathcal X_{\mathrm c})
=
Z:=\{p\}\times Y\simeq Y.
\]
Transversally to $Z$ the singularity is respectively a node and a cusp.  If $\operatorname{pr}_1$ and $\operatorname{pr}_2$ denote the two projections, then
\[
T^{\mathrm{vir}}\mathcal X_{\bullet}
=
\operatorname{pr}_1^*T^{\mathrm{vir}}X_{\bullet}
+
\operatorname{pr}_2^*TY
\qquad \text{in }K^0(\mathcal X_{\bullet}),
\]
so the virtual part is obtained from the curve calculation by exterior product with the smooth factor $Y$.

Write $\iota_{\mathrm n}:Z\hookrightarrow\mathcal X_{\mathrm n}$ and $\iota_{\mathrm c}:Z\hookrightarrow\mathcal X_{\mathrm c}$ for the closed immersions.  Under the standard identifications of supported Borel--Moore/Chow and $G$-theoretic groups with the corresponding theories on $Z\simeq Y$, exterior-product compatibility gives the distinguished supported representatives by multiplying the transverse defect by the characteristic class of $Y$; compare \cite{BrasseletSchuermannYokura,MaximSaitoSchuermann}.  Thus, as classes on the support $Z$,
\[
\widetilde\Delta^c_Z(\mathcal X_{\mathrm n})
=
-\,c(TY)\cap[Y],
\qquad
\widetilde\Delta^c_Z(\mathcal X_{\mathrm c})
=
-2\,c(TY)\cap[Y].
\]
Their images under forget support are obtained by applying $\iota_{\mathrm n*}$ and $\iota_{\mathrm c*}$, respectively.  Likewise the motivic Chern refinements on $Z$ are
\[
\widetilde\Delta^{mC_y}_Z(\mathcal X_{\mathrm n})
=
y\,\lambda_y(T^*Y)\cap[\mathcal O_Y],
\qquad
\widetilde\Delta^{mC_y}_Z(\mathcal X_{\mathrm c})
=
(y-1)\,\lambda_y(T^*Y)\cap[\mathcal O_Y],
\]
and the normalised motivic Riemann--Roch transformation $td_{(1+y)*}$ sends them to
\[
\widetilde\Delta^{T_y}_Z(\mathcal X_{\mathrm n})
=
y\,T_y^*(TY)\cap[Y],
\qquad
\widetilde\Delta^{T_y}_Z(\mathcal X_{\mathrm c})
=
(y-1)\,T_y^*(TY)\cap[Y].
\]
Unlike the plane-curve case, these are not concentrated in degree zero on the singular support.  If $h=c_1(\mathcal O_Y(1))$, the Euler sequence gives
\[
c(TY)\cap[Y]
=
(1+h)^{r+1}\cap[Y]
=
\sum_{j=0}^{r}\binom{r+1}{j}h^j\cap[Y],
\]
while its dual gives, in $K^0(Y)[y]$,
\[
(1+y)\lambda_y(T^*Y)
=
\bigl(1+y[\mathcal O_Y(-1)]\bigr)^{r+1}.
\]
Thus already the Chern defect has non-zero components in every even homological degree carried by $Z$, and the motivic Chern defect refines the same transverse node/cusp distinction inside the full $K$-theory of the singular stratum. 
Pushing the Hirzebruch defects to a point gives a compact numerical comparison.  Since
\[
\chi_y(\mathbb P^r)
=
\sum_{k=0}^{r}(-y)^k,
\]
one obtains
\[
\int_Z
\widetilde\Delta^{T_y}_Z(\mathcal X_{\mathrm n})
=
y\sum_{k=0}^{r}(-y)^k,
\qquad
\int_Z
\widetilde\Delta^{T_y}_Z(\mathcal X_{\mathrm c})
=
(y-1)\sum_{k=0}^{r}(-y)^k.
\]
At $y=-1$ these become
$
-(r+1),\ 
-2(r+1),
$
which are exactly the degrees of the Chern defects, since $\int_Y c(TY)=\chi(Y)=r+1$.  At $y=0$ the nodal defect again vanishes whereas the cuspidal defect has degree $-1$.  At $y=1$ the nodal degree is $1$ for $r$ even and $0$ for $r$ odd, while the cuspidal degree is zero.

This separates the transverse and longitudinal contributions: the node or cusp contributes the factors $-1,-2$, or $y,y-1$, while the singular stratum contributes $c(TY)$, $\lambda_y(T^*Y)$ or $T_y^*(TY)$.  In higher dimension the defect is therefore a characteristic class on the singular support, not merely a number attached to a singular point.  In these product examples the inclusion $Z=\{p\}\times Y\hookrightarrow X_{\bullet}\times Y$ has the projection to $Y$ as a retraction at the level of the standard homological and $G$-theoretic pushforwards; the forget-support map is consequently injective in these realisations, so the secondary ambiguity itself is trivial.

\medskip
\noindent\textbf{Collision and disappearance of the secondary ambiguity.}
The preceding examples separate the supported characteristic-class defect from its secondary ambiguity. The latter may disappear under degeneration when distinct components of the singular support collide.

Let $\Delta$ be a small disc with coordinate $t$.  In the weighted projective plane $\mathbb P(1,3,1)$, with coordinates $[X:Y:Z]$ of weights $(1,3,1)$, consider the family
\[
C_t=\left\{Y^2=X^3(X-tZ)^2(X-Z)\right\}.
\]
The defining equation is monic in $Y$, so the graded coordinate algebra is free of rank two over the polynomial algebra in $X$ and $Z$; in particular the family is flat over $\Delta$.  The singular point $[0:1:0]$ of $\mathbb P(1,3,1)$ lies on none of the fibres.  Moreover, for $t$ sufficiently small the polynomial on the right-hand side is not a square, hence each $C_t$ is irreducible and therefore connected.
For $t\neq0$, with $t$ sufficiently small, the affine equation in the chart $Z=1$ is
$
f_t(x,y)=y^2-x^3(x-t)^2(x-1).
$
It has precisely two singular points
$
q_t=(0,0),
\
p_t=(t,0).
$
At $q_t$ the right-hand side has a zero of order $3$, up to a unit, and at $p_t$ a zero of order $2$.  Hence $q_t$ is of type $A_2$ and $p_t$ of type $A_1$.  The two points at infinity are smooth: on the chart $X\neq0$, with $v=Y/X^3$ and $z=Z/X$, the equation is
$
v^2=(1-tz)^2(1-z),
$
and at $z=0$ one has $v=\pm1$.  At $t=0$ the affine equation becomes
$
f_0(x,y)=y^2-x^5(x-1),
$
so the two singular points collide at $q_0=(0,0)$ and produce an $A_4$-singularity.  Thus the family realises the collision
$
A_2+A_1\longrightarrow A_4,
$
which is the primitive collision for the minimal lifting considered by Kerner \cite{KernerCollisions}.  Notice that the collision is $\delta$-constant,
\[
\delta(A_2)+\delta(A_1)=1+1=2=\delta(A_4),
\]
whereas the total Milnor number jumps:
\[
\mu(A_2)+\mu(A_1)=2+1=3,
\qquad
\mu(A_4)=4.
\]
Let $Z_t=\operatorname{Sing}(C_t)_{\mathrm{red}}$.  For $t\neq0$,
$
H^{\mathrm{BM}}_0(Z_t;\mathbb Z)
=\mathbb Z[p_t]\oplus\mathbb Z[q_t],
\ 
H^{\mathrm{BM}}_0(C_t;\mathbb Z)\simeq\mathbb Z,
$
and the forget-support morphism is
$
\forg_t(a[p_t]+b[q_t])=(a+b)[\mathrm{pt}].
$
Consequently, in this degree-zero Borel--Moore realisation, exactness gives
$
\TLoc^c_{Z_t}(C_t)
=\ker(\forg_t)
=\mathbb Z\bigl([p_t]-[q_t]\bigr).
$
The Chern--Milnor defect is
$
\Delta_c(C_t)=-3[\mathrm{pt}],
$
and therefore
\[
\mathfrak{Mil}^c_{Z_t}(C_t)
=
\left\{a[p_t]+b[q_t]\mid a+b=-3\right\}.
\]
Vanishing cycles select the distinguished refinement
$
\widetilde\Delta^c_{\mathrm{vc},t}
=-[p_t]-2[q_t],
$
while every other refinement is uniquely of the form
\[
\widetilde\Delta^c_{s,t}
=(-1+s)[p_t]+(-2-s)[q_t],
\qquad s\in\mathbb Z,
\]
with
\[
\CSLoc_{Z_t}\bigl(
\widetilde\Delta^c_{\mathrm{vc},t},
\widetilde\Delta^c_{s,t}
\bigr)
=s\bigl([p_t]-[q_t]\bigr).
\]
The special fibre is different.  Its reduced singular support is the single point $
Z_0=\{q_0\},
$
and
\[
H^{\mathrm{BM}}_0(Z_0;\mathbb Z)
\longrightarrow
H^{\mathrm{BM}}_0(C_0;\mathbb Z)
\]
is an isomorphism.  Hence
\[
\TLoc^c_{Z_0}(C_0)=0.
\]
Moreover
$
\Delta_c(C_0)=-4[\mathrm{pt}],
$ and $
\mathfrak{Mil}^c_{Z_0}(C_0)
=\{-4[q_0]\}.
$
Thus the collision rigidifies the supported defect: the non-trivial torsor of refinements on a nearby fibre is replaced by a singleton on the special fibre.
Since the support varies with $t$, specialisation takes place between the corresponding supported groups rather than inside a single fixed group.  In this explicit family, the closures of the two sections $p_t$ and $q_t$ define a specialisation on the subgroup of degree-zero supported classes which they generate; both closures meet the special fibre at $q_0$, so
$
\operatorname{sp}_Z([p_t])=[q_0] $ and $
\operatorname{sp}_Z([q_t])=[q_0].
$
Therefore
$
\operatorname{sp}_Z\bigl([p_t]-[q_t]\bigr)=0,
$
so every secondary displacement dies under specialisation.  In particular
\[
\operatorname{sp}_Z\!\left(
\CSLoc_{Z_t}(\widetilde\Delta^c_{\mathrm{vc},t},
\widetilde\Delta^c_{s,t})
\right)=0.
\]
The primary defect, however, does not remain constant.  Indeed
$
\operatorname{sp}_Z(\widetilde\Delta^c_{\mathrm{vc},t})
=-3[q_0],
\ 
\widetilde\Delta^c_{\mathrm{vc},0}
=-4[q_0],
$
so that
\[
\widetilde\Delta^c_{\mathrm{vc},0}
-
\operatorname{sp}_Z(\widetilde\Delta^c_{\mathrm{vc},t})
=-[q_0].
\]
The extra term is the new vanishing contribution created by the collision.  Thus two independent phenomena occur at once: the secondary ambiguity disappears because there is no longer a pair of support components between which a fixed global defect may be redistributed, while the primary Milnor defect increases because the collision produces a more singular local germ.

The same support calculation applies to the homological Hirzebruch realisation.  If $R$ is the coefficient ring, then for $t\neq0$
\[
\TLoc^{T_y}_{Z_t}(C_t)
=R[y]\bigl([p_t]-[q_t]\bigr),
\]
whereas
\[
\TLoc^{T_y}_{Z_0}(C_0)=0.
\]
Hence every secondary Hirzebruch displacement is killed by the collision.  The distinguished nearby refinement is
\[
\widetilde\Delta^{T_y}_{\mathrm{vc},t}
=y[p_t]+(y-1)[q_t],
\]
while the special fibre is governed by the local $A_4$ vanishing-cycle class.  At $y=-1$ the preceding calculation recovers the jump from $-3[q_0]$ to $-4[q_0]$.  For motivic Chern classes the acting group on a nearby fibre remains the actual kernel
$
\ker\!\left(
G_0(Z_t)[y]\longrightarrow G_0(C_t)[y]
\right),
$
but on the special fibre the map $G_0(q_0)\to G_0(C_0)$ is injective, since pushforward to a point sends $[\mathcal O_{q_0}]$ to $1$.  Thus the motivic secondary ambiguity also vanishes on the collision fibre, although the nearby $K$-theoretic kernel need not coincide with the homological one.

This example isolates the geometric meaning of the secondary transgression group.  Its generator $[p_t]-[q_t]$ does not measure the amount of singularity: it measures the freedom to redistribute a fixed global defect between distinct components of its support.  When those components collide, that freedom disappears.  The surviving change is primary and is measured instead by the new vanishing cycles of the limiting singularity.

\section*{Acknowledgements}

The authors are partially supported by the Università degli Studi di Bari Aldo Moro. M.C. would like to thank Jean-Paul Brasselet, Alana Cavalcante, Fernando Lourenço, and José Seade for their valuable suggestions and stimulating discussions on localisation and characteristic class theories.

\subsection*{AI disclosure}
AI were used for proofreading, improving the exposition, and checking computations. All mathematical content, arguments, and results are the authors' own; the authors have independently verified the final manuscript and assume full responsibility for its content.

\end{document}